\documentclass[11pt,reqno]{amsart}
\allowdisplaybreaks

\usepackage[english]{babel}
\usepackage[latin1]{inputenc}
\usepackage[T1]{fontenc}
\usepackage{amssymb,amsmath,amstext,amsfonts,amsthm}
\usepackage{mathtools}
\usepackage{verbatim}
\usepackage{relsize}
\usepackage{graphicx,color}
\usepackage[dvipsnames]{xcolor}
\usepackage[all]{xy}
\usepackage[colorlinks=true, urlcolor=blue, linkcolor=blue, citecolor=blue]{hyperref}
\usepackage{fancyvrb}
\usepackage{tikz-cd}
\usepackage{comment}
\usepackage{hhline}
\usepackage{diagbox}
\usepackage{epigraph}
\usepackage[left=25mm, right=25mm, top=30mm, bottom=30mm]{geometry}
\usepackage{ytableau}
\usepackage{cleveref}
\usepackage{thm-restate}
\usepackage{setspace}
\usepackage{indentfirst}
\numberwithin{equation}{section}

\theoremstyle{plain}
\newtheorem{Theorem}{Theorem}[section]
\newtheorem*{Theorem*}{Theorem}
\newtheorem{Corollary}[Theorem]{Corollary}
\newtheorem{Lemma}[Theorem]{Lemma}
\newtheorem{Proposition}[Theorem]{Proposition}
\newtheorem{Conjecture}[Theorem]{Conjecture}
\newtheorem{Algorithm}[Theorem]{Bott's Algorithm}

\theoremstyle{definition}
\newtheorem{Definition}[Theorem]{Definition}

\newtheorem*{Notation}{Notation}

\theoremstyle{remark}
\newtheorem{Remark}[Theorem]{Remark}

\newcommand{\C}{\mathbb{C}}
\newcommand{\NN}{\mathbb{N}}

\newcommand{\Z}{\mathbb{Z}}
\newcommand{\PP}{\mathbb{P}}
\newcommand{\F}{\mathbb{F}}
\newcommand{\R}{\mathbb{R}}
\newcommand{\mR}{\mathsmaller{\mathbb{R}}}

\newcommand{\xx}{\textbf{x}}
\newcommand{\yy}{\textbf{y}}

\newcommand{\nn}{\textbf{n}}

\newcommand{\sE}{\mathcal{E}}

\newcommand{\sI}{\mathcal{I}}
\newcommand{\sS}{\mathcal{S}}
\newcommand{\sF}{\mathcal{F}}

\newcommand{\sQ}{\mathcal{Q}}

\newcommand{\sO}{\mathcal{O}}
\newcommand{\SL}{\mathrm{SL}}

\newcommand{\ser}{\mathcal E^{(r)}}

\makeatletter
\@namedef{subjclassname@2020}{%
	\textup{2020} Mathematics Subject Classification}
\makeatother

\usepackage{tcolorbox}
\newtcolorbox{mybox}{colback=blue!5!white,colframe=blue!75!black}


\title{On defective spans of singular vector tuples beyond the boundary format}

\author[E.T. Turatti]{Ettore Teixeira Turatti}
\address[E.T. Turatti]{Faculty of Mathematics, Informatics, and Mechanics, University of Warsaw, Banacha 2, 02-097 Warsaw, Poland}
\email{e.teixeira-turatti@uw.edu.pl}

\author[E. Ventura]{Emanuele Ventura}
\address[E. Ventura]{Politecnico di Torino, Dipartimento di Scienze Matematiche ``G. L. Lagrange'', Corso Duca degli Abruzzi 24, 10129 Torino, Italy}
\email{emanuele.ventura@polito.it}

\begin{document}

\maketitle

\begin{abstract}
In this paper, we study tensor spaces beyond the boundary format and analyze whether the general critical space coincides with the general span of singular vector tuples. For all tensor spaces exceeding the boundary format by one in an arbitrary number of factors, we relate the codimension of this span within the critical space to the dimension of the kernel of a map in cohomology. Furthermore, we exhibit an infinite family of order-three tensors with a {\it defective} behavior: the general span of singular vector tuples achieves the maximum possible codimension rather than the expected minimum. Finally, we conjecture a classification of the behavior of critical spaces in this regime and draw a connection to Koszul cohomology. 
\end{abstract}

\section{Introduction}\label{sec: intro}

The spectral theory of tensors is a deep subject both in its applied and algebraic geometry aspects. A fundamental notion in this setting is that of 
a {\em singular vector $k$-tuple} of an order-$k$ tensor; this is the generalization of the notion of singular pair of a rectangular matrix. Singular vector $k$-tuples retain important properties of singular pairs. For instance, in the problem of minimizing the distance between a target rank-one tensor and a given tensor $T$, the singular vector $k$-tuples of $T$ correspond to the constrained critical points of the distance function. Therefore, they essentially provide an answer to the so-called {\em best rank-one approximation problem} for $T$ \cite{L}.

Beyond approximation problems, singular vector $k$-tuples have been studied from several geometric perspectives. In particular, their relation to tensor decompositions: in the symmetric setting, \cite{HT25} investigates when a form admits a minimal decomposition via its singular vector $k$-tuples, while \cite{RHST25} extends this question to the general setting, including non-minimal decompositions. Another direction of research is exploring structural properties such as orthogonality \cite{RSZ}, and power methods algorithms for decomposing tensors \cite{WPKS}.

One of the fundamental properties of singular pairs of matrices is that they yield a minimal decomposition as a sum of rank-one matrices: the singular value decomposition (SVD). However, this property does not generalize to singular vector $k$-tuples of higher-order tensors. In view of this fundamental difference, Ottaviani and Paoletti \cite[Section 5.2]{OP} introduced the \emph{singular space} $H_T$ of a tensor $T$ -- later called the \emph{critical space} \cite[Definition 2.8]{DOT} --  as the vanishing locus of the equations defining singular vector $k$-tuples, however not restricted to rank-one tensors. An important feature of this space is that $T\in H_T$ and $\langle Z_T\rangle\subset \PP(H_T)$, where $Z_T$ is the variety of singular vector $k$-tuples of $T$; see Definition \ref{def:Z_T}. Draisma, Ottaviani and Tocino showed that, when the tensor space has boundary format (see Definition \ref{def: boundary format}), the equality $\langle Z_T\rangle= \PP(H_T)$ holds \cite[Theorem 1.1]{DOT}. As a consequence, $T$ may be decomposed (perhaps in a non-minimal way) as a sum of its singular vector $k$-tuples. This is a generalization of the celebrated Eckart-Young theorem for matrices. 

The boundary format has an important geometric meaning: it corresponds to the cases where the Segre variety is dual defective, i.e., the dual variety is not a hyperplane, see \cite[Corollary 5.11]{GKZ}. Moreover, in the context of \cite{DOT}, this format is a technical condition that guarantees the vanishing of cohomologies of the bundle $\bigwedge^r \mathcal E^\ast\otimes \mathcal O(d_1,\dots,d_k)$, where $\mathcal E$ is the \emph{Friedland-Ottaviani bundle}; see \Cref{def:FO-bundle}. In \cite{ST22}, the authors explore what happens when the boundary format is disregarded for almost binary tensors. Although the vanishing results do not hold, the authors show that  for almost binary tensors, up to a finite list of cases, one still has the equality $\langle Z_T\rangle=\PP(H_T)$ \cite[Theorem 5.7]{ST22}. Moreover, \cite[Theorem 6.6]{ST22} shows that it is sufficient that $T\in \langle Z_T\rangle$ for a generic tensor $T$ to be determined by its singular vector $k$-tuples, thus extending \cite[Theorem 1.2]{turatti2021tensors} which also relied on the assumption of boundary format.

In this work, motivated by the quest of getting rid of the boundary format hypothesis, we focus on the case when $V=\C^{n_1+1}\otimes \dots \otimes \C^{n_k+1}$, and $n_k=1+\sum_{i=1}^{k-1} n_i$, i.e., the minimal case beyond boundary format. We perform a thorough analysis of the cohomological vanishings required via Bott-Borel-Weil theorem, and show that the picture is peculiar for the case of three-factor tensors ($k=3$), where there are fewer vanishings, while for higher-order ($k\geq 4$), the vanishings are uniform, see Propositions \ref{lemma: gen coh 1}, \ref{lemma: gen coh 2} and \ref{lemma: gen coh 3}. 

In \Cref{prop:dim k+1-factors} and \Cref{prop: dim k=3 alpha} we show that the codimension of $\langle Z_T\rangle$ in $\PP(H_T)$ amounts to the dimension of the kernel of a map of cohomologies $\alpha_T$. This map is later spelled out in \Cref{cor:spelloutalpha}, formalizing some steps in \cite{ST22} which were not explicit. In \Cref{sec:inequality} we show that in only few cases this map cannot be injective due to dimensional count, these include the cases where $\langle Z_T\rangle\subsetneq \PP(H_T)$ in \cite{ST22}. The map $\alpha_T$ is full-rank in all experiments we performed, besides one specific family. Our main result actually identifies a family where the map $\alpha_T$ vanishes identically, and $\langle Z_T\rangle$ has the maximal codimension possible. Namely:
 \begin{restatable*}{Theorem}{mainthm}\label{Thm: (2,n,n+2)}
    Let $T\in V=\C^2\otimes \C^n\otimes \C^{n+2}$ be a general tensor.  Then $\langle Z_T\rangle\subset \PP(H_T)$ has codimension $n-1$.
\end{restatable*}
\noindent This gives a {\it defective} case for $\langle Z_T\rangle$, i.e., its dimension is lower than the maximal possible from parameter counting. Apart from this case, we conjecture that $\alpha_T$ has maximal rank. Equivalently, we conjecture that $\langle Z_T\rangle= \PP(H_T)$, up to finitely many cases; see Conjectures \ref{conj:2} and \ref{conj:3}. In Section \ref{sec:Koszul coh approach}, we draw a connection between these conjectures and the vanishing of certain Koszul cohomology modules.

\medskip

\noindent {\bf Acknowledgements.} 
This project started during the semester program Algebraic Geometry with Applications to Tensors and Secants (AGATES), held in Warsaw in 2022. We warmly thank Claudiu Raicu for explanations and discussions that led to write \Cref{sec:Koszul coh approach}. We warmly thank Joachim Jelisiejew, Luca Sodomaco and Jerzy Weyman for useful discussions at various stages of this project. E.T.T. acknowledges support by the National Science Center, Poland, under the project ``Tensor rank and its applications to signature tensors of path'', 2023/51/D/ST1/02363. E.V. is a member of the research group GNSAGA of INdAM and acknowledges that the preparation of this article was partially funded by the Italian Ministry of University and Research, through the project {\it Applied Algebraic Geometry of Tensors}, PRIN 2022 Protocol no. 2022NBN7TL.

\section{Tensors and singular vector tuples}\label{sec: prelim}

\begin{Notation}
We use $[k]$ to denote the set of indices $\{1,\ldots,k\}$. We denote by $\xx$ a vector $(x_1,\ldots,x_k)$ of $k$ variables and $|{\xx}|\coloneqq x_1+\cdots+x_k$. Define ${\bf 1} = (1,\ldots,1)\in \NN^k$ and, for $m\in \NN$, let $m{\bf 1} = (m,\ldots, m)\in \NN^k$. 
\end{Notation}

For every $i\in[k]$ we consider an $(n_i+1)$-dimensional vector space $V_i$ over the field $\F=\R$ or $\F=\C$. When $\F=\R$, we will use the notation $V_i^\mR$. On the other hand, $V_i$ is be assumed to be over $\C$. We denote by $V$ the tensor product $\bigotimes_{i=1}^k V_i$. This is the vector space of {\em order $(k+1)$ tensors of format $\nn=(n_1+1,\ldots,n_k+1)$}.

\begin{Definition}\label{def: Segre embedding}
A tensor $T\in V$ is of {\it rank one} if $T = \xx_1\otimes\cdots\otimes \xx_k$ for some non-zero vectors $\xx_{j}\in V_j$ for all $j\in[k]$. Rank-one tensors are parametrized by the Segre variety of format $n_1\times\dots\times n_k$, that is, the image of the projective morphism
$$
\mathrm{Seg}\colon \PP(V_1)\times \cdots \times \PP(V_k)\to\PP(V),\ \left([\xx_1],\dots,[\xx_k]\right)\mapsto [\xx_1\otimes\dots\otimes\xx_k].
$$

\end{Definition}

Throughout the paper, we adopt the shorthand $\PP=\PP(V_1)\times \cdots \times \PP(V_k)$ to indicate the Segre variety. Moreover, we abuse notation identifying a tensor $T\in V$ with its class in the projective space $\PP(V)$. On each projective space $\PP(V_i)$ we fix a quadric $q_i\in\C[x_{i,1},\ldots,x_{i,n_i+1}]_2$ associated to a positive definite real quadratic form $q_i^\mR\colon V_i^\mR\to\R$.

\begin{Definition}\label{def: Frobenius inner product}
The {\it Frobenius (or Bombieri-Weyl) inner product} of two complex decomposable tensors $T = \xx_1\otimes\cdots\otimes \xx_k$ and $T' = \yy_1\otimes\cdots\otimes \yy_k$ of $V$ is
\begin{equation}\label{eq: Frobenius inner product for tensors}
q_F(T, T')\coloneqq q_1(\xx_1, \yy_1)\cdots q_k(\xx_k, \yy_k)\,,
\end{equation}
and it is naturally extended by linearity to every vector in $V$.
We identify all the vector spaces with their duals using the Frobenius inner product.
\end{Definition}

\begin{Definition}\label{def: singular vector tuple}
Let $T\in V$. A {\em singular vector $k$-tuple} of $T$ is a $k$-tuple $(\xx_1,\ldots,\xx_k)$ of nonzero vectors $\xx_i\in V_i$ such that
\begin{equation}\label{eq: def singular vector tuple matrix}
\mathrm{minors}_{2\times2}
\begin{bmatrix}
T(\xx_1\otimes\cdots\otimes \xx_{i-1}\otimes \xx_{i+1}\otimes\cdots\otimes \xx_k)\\
\xx_i
\end{bmatrix}=0
\quad\forall\,i\in[k]\,,
\end{equation}
where
\begin{equation}\label{eq: def contraction}
T(\xx_1\otimes\cdots\otimes \xx_{i-1}\otimes \xx_{i+1}\otimes\cdots\otimes \xx_k)\coloneqq\sum_{j_\ell\in[n_\ell]}t_{j_1\cdots\,j_i\cdots\,j_k}\,x_{1,j_1}\cdots \widehat{x_{i,j_i}}\cdots x_{k,j_k}
\end{equation}
is the tensor contraction of $T=(t_{j_1,\ldots,j_k})$ with respect to $\xx_1\otimes\cdots\otimes \xx_{i-1}\otimes \xx_{i+1}\otimes\cdots\otimes \xx_k$. The symbol $\widehat{x_{i,j_i}}$ in \eqref{eq: def contraction} means that the variable $x_{i,j_i}$ is omitted in the product. 
If we interpret $T$ as a multi-homogeneous polynomial in the coordinates of each vector $\xx_i$, then the previous tensor contraction corresponds to the gradient $\nabla_iT$ with respect to the vector $\xx_i=(x_{i,1},\ldots,x_{i,n_i+1})$.
\end{Definition}

The Euclidean Distance degree (ED-degree) of the Segre variety is defined as the number of critical points of the squared distance function \( q_F(T - x, T - x) \), where \(T\) is a general tensor and \(x\) ranges over the Segre variety. It was shown in \cite{L} that these critical points coincide with the singular vector \(k\)-tuples of \(T\). Consequently, the ED-degree amounts to the number of singular vector \(k\)-tuples of a general tensor.

\begin{Theorem}{\cite[Theorem 1]{FO}}\label{thm: FO formula}
Let $T\in V$ be a general tensor. Then $T$ has exactly $\mathrm{ED}(\nn)$ simple singular vector
tuples, where $\mathrm{ED}(\nn)$ equals the coefficient of the monomial $h_1^{n_1}\cdots h_k^{n_k}$ in the polynomial
\[
\prod_{i=1}^k\frac{\widehat{h}_i^{n_i+1}-h_i^{n_i+1}}{\widehat{h}_i-h_i},\quad\widehat{h}_i\coloneqq\sum_{j\neq i}^k h_j\,.
\]
\end{Theorem}

We refer to \cite{DHOST} for more details on ED-degrees of algebraic varieties. 

\begin{Definition}\label{def: boundary format}
Consider a tensor space $V$ of format $\nn=(n_1+1,\ldots,n_k+1)$. Then $\nn$ is
\begin{enumerate}
    \item a {\em sub-boundary format} if for all $i\in[k]$ we have $n_i\le \sum_{j\neq i}n_j$.
    \item a {\em boundary format} if for some $i\in[k]$ we have $n_i= \sum_{j\neq i}n_j $.
    \item we have {\em beyond boundary format} if for some $i\in[k]$ the condition $n_i > \sum_{j\neq i}n_j$ holds.
\end{enumerate}
\end{Definition}

Next, we introduce the main tensor locus studied in this paper.

\begin{Definition}\label{def:Z_T}
Let \(T \in V\). We define
\begin{equation}\label{eq:def_Z_T}
Z_T \coloneqq \left\{ [\xx_1 \otimes \cdots \otimes \xx_k] \in \PP(V) \,\middle|\, (\xx_1,\ldots,\xx_k) \text{ is a singular vector \(k\)-tuple of } T \right\}.
\end{equation}
\end{Definition}

For a general tensor \(T \in V\), the set \(Z_T\) is zero-dimensional, and its cardinality coincides with the ED-degree of the Segre variety, computed by the Friedland--Ottaviani formula of \Cref{thm: FO formula}.

Throughout the paper, we compare the projective span $\langle Z_T\rangle$ with another important linear subspace associated to singular vector $k$-tuples which we recall in the next definition.

\begin{Definition}\label{def: critical space}
The {\it critical space} $H_T$ of a tensor $T\in V$ is the linear subspace of $V$ defined by the equations (in the unknowns $z_{i_1\cdots i_k}$ that serve as linear functions on $V$)
\begin{equation}\label{eq: equations critical space}
\sum_{i_\ell\in[n_\ell]}\left(t_{i_1\cdots\,p\,\cdots\,i_k}\,z_{i_1\cdots\,q\,\cdots\,i_k}-t_{i_1\cdots\,q\,\cdots\,i_k}\,z_{i_1\cdots\,p\,\cdots\,i_k}\right)=0\quad\text{where $1\le p<q\le n_\ell$ and $\ell\in[k]$.}
\end{equation}
\end{Definition}

The equations in \eqref{eq: equations critical space} are obtained after computing the two by two minors of the matrix in \eqref{eq: def singular vector tuple matrix} and substituting the relations $z_{j_1\cdots j_k}=x_{1,j_1}\cdots x_{k,j_k}$. In particular, the equations of $H_T$ are linear relations among the elements of $Z_T$, thus in general $\langle Z_T\rangle\subset\PP(H_T)$. Another immediate observation is that $T$ always belongs to $H_T$, as its coordinates always satisfy the equations in \eqref{eq: equations critical space}. We recall a result on the dimension of the critical space. 

\begin{Proposition}{\cite[Proposition 5.6]{OP}}

Consider a tensor $T\in V$ of format $\nn=(n_1+1,\ldots,n_k+1)$. Assume $n_1\le\cdots\le n_k$ and let $N=\prod_{i=1}^{k-1}(n_i+1)$.
The dimension of the critical space $H_T\subset V$ is
\begin{equation}
\begin{cases}
\prod_{i=1}^k(n_i+1)-\sum_{i=1}^k\binom{n_i+1}{2} & \text{for $(n_k+1)\le N$}\\
\binom{N+1}{2}-\sum_{i=1}^{k-1}\binom{n_i+1}{2} & \text{for $(n_k+1)\ge N$}\,.
\end{cases}
\end{equation}
\end{Proposition}

\begin{Proposition}{\cite[Proposition 3.6]{DOT}}\label{prop: equality span critical sub-boundary format}
Let $T$ be general in a tensor space $V$ of sub-boundary or boundary format. Then $\langle Z_T\rangle=\PP(H_T)$.
\end{Proposition}

The containment $\langle Z_T\rangle \subset \mathbb{P}(H_T)$ can be strict when the tensor space is beyond boundary format \cite[Remark 3.7]{DOT}. This leads the authors of {\it loc. cit.} to pose the natural problem of studying the dimension of $\langle Z_T\rangle$ beyond boundary format.

\section{Vector bundles, cohomology and representation theory}\label{sec: Cohomology}

Here we recall the main techniques we shall apply throughout the paper. We refer to \cite{Weyman} for more details.

\begin{Notation}
For every $i\in[k]$, we consider the projection $\pi_i\colon\PP\to\PP(V_i)$. Moreover, we denote by $\mathcal{Q}_i$ the quotient bundle on $\PP(V_i)$. The fiber of $\mathcal{Q}_i$ over $[\xx_i]\in\PP(V_i)$ is $V_i/\langle\xx_i\rangle$.
\end{Notation}

\begin{Theorem}[K\"unneth's formula]\label{thm: Kunneth}
Let $q$ be a nonnegative integer. Consider the Segre product $\PP$ and let $\mathcal B_i$ be a vector bundle on $\PP(V_i)$ for all $i\in[k]$. Then
\begin{equation}\label{eq: Kunneth}
H^q\bigg(\PP,\bigotimes_{i=1}^k \pi_i^\ast \mathcal B_i\bigg)\cong \bigoplus_{|{\bf j}|=q}\bigotimes_{i=1}^k H^{j_i}(\PP(V_i), \mathcal B_i)\,.
\end{equation}
\end{Theorem}

Let $G$ be a semisimple simply connected group, let $P\subset G$ be a parabolic subgroup. Let $\Phi^+$ be the (finite) set of positive roots of $G$. Let $\delta=\sum \lambda_i$ be the sum of all the fundamental weights and let $\lambda$ be a weight. Let $E_\lambda$ be the homogeneous bundle arising from the irreducible representation of $P$ with highest weight $\lambda$ and $(\ ,\ )$ be the Killing form.
\begin{Definition}
The {\em fundamental Weyl chamber} is the convex set
\begin{equation}\label{eq: Weyl chamber}
\Lambda=\{\lambda\text{ is a weight} \mid (\lambda,\alpha)\geq0,\  \forall\ \alpha\in \Phi^+ \}.
\end{equation}
\end{Definition}

\begin{Theorem}
Let $G$ be a semisimple and simply-connected group and let $\lambda: B\rightarrow \mathcal C^{*}$ be a weight, where $B$ is a Borel subgroup of $G$. Let $P\subset G$ be a parabolic subgroup and let $\rho$ be the representation of $P$ with highest weight $\lambda$. This defines a line bundle $L_{\lambda}$ on the flag variety $G/B$. Denote by $\pi: G/B\rightarrow G/P$ the natural projection, which is flat with rational fibers. Let $\mathcal E$ be the homogeneous bundle defined by the irreducible representation $\rho$ of $P$ \cite[Chapter 10]{Ottaviani95}. Then 
\[
\pi_{*} L_{\lambda} \cong \mathcal E. 
\]
\end{Theorem}

\begin{Definition}
The vector bundle $\mathcal E$ on $G/P$ above is also denoted $\mathcal E_{\lambda}$, where $\lambda$ is the {\it weight} 
corresponding to $\mathcal E$. 
\end{Definition}

\begin{Definition}
The weight $\lambda$ is called {\em singular} if there exists a root $\alpha \in \Phi^+$ such that $(\lambda,\alpha)=0$. Otherwise, if $(\lambda,\alpha)\neq 0$ for all the roots $\alpha \in \Phi^+$, we say that $\lambda$ is {\em regular of index $p$} if there exist exactly $p$ roots $\alpha_1,\dots,\alpha_p\in\Phi^+$ such that $(\lambda,\alpha)<0$.
\end{Definition}

\begin{Theorem}[Bott-Borel-Weil]\label{thm: Bott}
The following statements hold: 
\begin{enumerate}
    \item If $\lambda + \delta$ is singular, then $H^i(G/P,\mathcal E_\lambda)=0$ for all $i$.
    \item If $\lambda+\delta $ is regular of index $p$, then $H^i(G/P, \mathcal E_\lambda)=0$ for $i\neq p$. Moreover $H^p(G/P, \mathcal E_\lambda)=G_{\sigma(\lambda+\delta)-\delta}$, where $\sigma(\lambda+\delta)$ is the unique element of the fundamental Weyl chamber of G which is congruent to $\lambda+\delta$ under the action of the Weyl group and $G_{\sigma(\lambda+\delta)-\delta}$ is the corresponding irreducible $G$-representation. 
\end{enumerate}
\end{Theorem}

\begin{Algorithm}\cite[Remark 4.1.5]{Weyman}
We expose how to find the cohomology groups $H^i(G/P, \mathcal E_\lambda)$ in the case when $G = \mathrm{SL}(n+1)$.
Let $L_i - L_{i+1}$ for $1\leq i\leq n$ be the positive roots and let $\mathbb Z^{n+1}$ be the lattice spanned by them; recall that $L_1+\cdots + L_{n+1}=0$. 
Thus the dual of the Lie algebra of the Cartan subalgebra $\mathfrak h^{*}\cong \mathbb Z^{n+1}/\mathbb Z(1,\ldots, 1)$. The $k$th fundamental weight is
the sum $\lambda_k = \sum_{i=1}^k L_i$. Let $\lambda\in \mathbb Z^{n+1} = \mathbb Z\langle L_1,\ldots, L_{n+1}\rangle$ be $\lambda = \sum_{i=1}^{n+1} \alpha_i L_i$. In this terminology, $\delta = (n,n-1,\ldots,1, 0)\in \mathbb Z^{n+1}$. The Weyl group $\mathfrak S_{n+1}$ naturally acts permuting the entries. 
For an element $\sigma\in \mathfrak S_{n+1}$ define the action
\begin{equation}\label{exchanges}
\sigma^{\bullet} (\lambda) = \sigma(\lambda+\delta) - \delta. 
\end{equation}
For $\sigma_i = (i, i+1)\in \mathfrak S_{n+1}$, one has $\sigma_i^{\bullet}(\lambda) = (\alpha_1, \ldots, \alpha_{i-1}, \alpha_{i+1}-1, \alpha_{i}+1, \alpha_{i+2},\ldots, \alpha_{n+1})$. We have the following cases: 
\begin{enumerate}
\item[(i)] If $\lambda$ is a partition, i.e. $\alpha_{i+1}\geq \alpha_{i}$ for every $i$, then $\lambda+\delta$ is regular of index $0$. Therefore,
one has $H^0(\mathrm{SL}(n+1)/P, \mathcal E_\lambda)=\mathrm{SL}(n+1)_{w(\lambda+\delta)-\delta} = \mathrm{SL}(n+1)_{\lambda}$. 
 
\item[(ii)] If $\lambda$ is not a partition, apply exchanges $\sigma_i$ acting by \eqref{exchanges}. Then, after each exchange, two possibilities may occur: 
\begin{enumerate}
\item[(a)] One has $\alpha_{i+1} = \alpha_{i}+1$ for some index $i$. Then the element $\lambda+\delta$ is singular and so $H^i(\mathrm{SL}(n+1)/P, \mathcal E_\lambda)=0$ for every $i$. 

\item[(b)] Otherwise, we have reached a non-increasing sequence $\beta = \sigma(\lambda+\delta)-\delta$, after having performed $p>0$ exchanges. Then 
$\lambda$ is regular of index $p$. Hence $H^p(\mathrm{SL}(n+1)/P, \mathcal E_\lambda)=\mathrm{SL}(n+1)_{\beta}$ and $H^i(\mathrm{SL}(n+1)/P, \mathcal E_\lambda) = 0$ for $i\neq p$. 
\end{enumerate}
Note that $\beta\in \mathbb Z^{n+1}$ is a non-increasing sequence, but it may have negative entries. However, $\beta$ is an element of $\mathfrak h^{*}$ and so it is in the equivalence class of the non-negative non-increasing weight $\beta'=\beta + (-\beta_{n+1},\ldots, -\beta_{n+1})$. The weight $\beta'$ is in the fundamental Weyl chamber $\Lambda$. 
\end{enumerate}
\end{Algorithm}

Let $\Omega^r_{\PP^n}$ be the $\sO(k)$-twisted vector bundle of differential $r$-forms on $\PP^n$. 

\begin{Proposition}{\cite[Proposition 4.2.3]{Weyman}}
Let $P=P(\alpha_1)$ be the parabolic subgroup of $G=\mathrm{SL}(n+1)$ corresponding to omitting the first simple root. 
Then a homogeneous vector bundle $\mathcal E_{\lambda}$ on $\PP^n = \mathrm{SL}(n+1)/P(\alpha_1)$ is the push-forward 
of a line bundle $L_{\lambda}$ of weight $\lambda$ on the flag variety $\mathrm{SL}(n+1)/B$. One has canonical isomorphisms $\Omega^r_{\PP^n}(r)\cong \bigwedge^{r}\sQ^{*}_{\PP^n}\cong \bigwedge^{n-r}\sQ_{\PP^n}(-1)$. The weight assigned to $\bigwedge^{n-r}\sQ_{\PP^n}\cong \Omega^r_{\PP^n}(r+1)$
is $\lambda_{r+1}$ for $0\leq r\leq n-1$.
\end{Proposition}

\begin{Remark}
The following isomorphism holds: 
\[
\mathcal E_{\lambda_1}\otimes \mathcal E_{\lambda_2}\cong \mathcal E_{\lambda_1+\lambda_2}.
\]
Thus, for any $k\in \mathbb Z$, the vector bundle $\Omega^r_{\PP^n}(r+1+k)\cong \Omega^r_{\PP^n}(r+1)\otimes \mathcal O_{\PP^n}(k)$ has corresponding weight $k\cdot \lambda_{1}+\lambda_{r+1}$. 
\end{Remark}

The next proposition is a direct consequence of Bott-Borel-Weil theorem for $G = \mathrm{SL}(n+1)$. 

\begin{Proposition}[Bott's Theorem]\label{bott form}
For $q,n,r\in \mathbb Z_{\geq 0}$ and $k\in \mathbb Z$, one has: 
\begin{equation}\label{eq: coh Omega}
\dim_{\mathbb C} H^q(\Omega_{\PP^n}^r(k)):=h^q(\Omega_{\PP^n}^r(k))=\begin{cases}
    \binom{k+n-r}{k}\binom{k-1}{r}\text{ if $q=0\leq r\leq n$ and $k>r$};\\
    1 \text{ if $0\leq q=r\leq n$ and $k=0$};\\
    \binom{-k+r}{-k}\binom{-k-1}{n-r} \text{ if $q=n\geq r\geq 0$ and $k<r-n$};\\
    0 \text{ otherwise.}
    \end{cases}
\end{equation}
\end{Proposition}

\begin{Definition}\label{def:FO-bundle}
    We now introduce the central vector bundle construction for dealing with singular vector tuples. The {\it Friedland-Ottaviani vector bundle} on $\PP$ is defined by

\begin{equation}
\sE\coloneqq\bigoplus_{i=1}^k\sE_i\,,\quad\sE_i\coloneqq(\pi_i^\ast\mathcal{Q}_i)\otimes\mathcal{O}(1,\ldots,1,\overbracket{0}^i,1,\ldots,1)\quad\forall\,i\in[k]\,.
\end{equation}

\end{Definition}

We have $d:=\mathrm{rank}(\sE)=\dim(\PP)=\sum_{i=1}^k n_i$. For every $i\in[k]$, the tensor $T$ yields a global section of $\sE_i$, which over the point $([\xx_1],\ldots,[\xx_k])\in\PP$ is the map
\[
(\lambda_1\xx_1,\ldots,\lambda_k\xx_k)\in\prod_{i=1}^k\langle\xx_i\rangle \mapsto [T(\lambda_1\xx_1\otimes\cdots\otimes\lambda_{i-1}\xx_{i-1}\otimes\lambda_{i+1}\xx_{i+1}\otimes\cdots\otimes\lambda_k\xx_k)]\in{V_i}/{\langle\xx_i\rangle}\,.
\]
Combining these $k$ sections, the tensor $T$ yields a global section $s_T$ of $\sE$. The tuple $(\xx_1,\ldots,\xx_k)$ is a singular vector $k$-tuple of $T$ if and only if the point $([\xx_1],\ldots,[\xx_k])$ is in the zero locus of the section $s_T$.
The section $s_T$ of $\sE$ yields a homomorphism $\sE^\ast\to\sO$ of sheaves whose image is contained in the ideal sheaf $\sI_{Z_T}$ of the zero locus of $s_T$.

Since the ideal sheaf $\sI_{Z_T}$ is zero-dimensional for a generic choice of $T$, it has a resolution given by the Koszul complex
\begin{equation}\label{eq: Koszul}
0\to \bigwedge^{d}\sE^\ast\xrightarrow{\varphi_{d}}\bigwedge^{d-1}\sE^\ast\xrightarrow{\varphi_{d-1}}\cdots\xrightarrow{\varphi_{2}}\sE^\ast\to \sI_{Z_T}\to 0.\,
\end{equation}

We will split the previous complex into short exact sequences, therefore we define $\sF_{i+1}\coloneqq\left(\bigwedge^{i}\sE^\ast\right)/\mathrm{Im}(\varphi_{i+1})$ and use the shorthand $\sF_i({\bf 1})$ to denote the tensor product $\sF_i\otimes\sO({\bf 1})$. For all integer $r\ge 1$ we use the shorthand $\sE^{(r)}$ to denote $\left(\bigwedge^r\sE^\ast\right)\otimes\sO({\bf 1})$.

\begin{Remark}\label{lem: interpretation of H^0(E^1)}
We recall that it follows from \cite[Section 3]{DOT} that the equations cutting the linear space $\PP(H_T)$ are given by the sections in $H^0(\sE^{(1)})$. 
\end{Remark}

Let  $T\in V$ be a general $(k+1)$-order tensor  of format $\nn=(n_1+1,\ldots,n_{k+1}+1)$. From the Koszul complex \eqref{eq: Koszul} and tensoring  with $\sO({\bf 1})$, for $n=\sum_{i=1}^{k+1} n_i$ we obtain the short exact sequences 
\begin{equation}\label{eq: shortexactsequence}
\begin{gathered}
0\to \sF_{r+1}({\bf 1})\to \sE^{(r)}\to \sF_r({\bf 1})\to 0\  \  \mbox{ for } 2\leq r\leq n, \sF_{d+1}({\bf 1})=0, \mbox{ and} \\
0\to \sF_{2}({\bf 1})\to \sE^{(1)}\to \sI_{Z_T}({\bf 1})\to 0\,.
\end{gathered}
\end{equation}
The long exact sequence in cohomology applied to the short exact sequences above is: 
\begin{equation}\label{longcoh}
\begin{gathered}
\cdots\to H^{r-2}(\sE^{(r)})\to H^{r-2}(\sF_r({\bf 1}))\to H^{r-1}(\sF_{r+1}({\bf 1}))\to H^{r-1}(\sE^{(r)})\to\\
\to H^{r-1}(\sF_r({\bf 1}))\to H^{r}(\sF_{r+1}({\bf 1}))\to H^{r}(\sE^{(r)})\to\cdots
\end{gathered}
\end{equation}
Equipped with this notation, we state useful results about some vanishing of cohomologies of $\sE^{(r)}$.

\begin{Lemma}{\cite[Lemma 3.4]{DOT}}\label{lemma 3.2 DOT}
Recall that $\Omega^r_{\PP(W)}(k)$ is the $\sO(k)$-twisted vector bundle of differential $r$-forms
on $\PP(W)$. One has the following isomorphism:
\begin{equation}\label{eq: iso omega}
\sE^{(r)}\cong\bigoplus_{r_1+\cdots+r_k=r}\bigotimes_{i=1}^k\pi_i^\ast \Omega_{\PP(V_i)}^{r_i}(2r_i+1-r)\,.
\end{equation}
\end{Lemma}
\begin{Proposition}\label{lemma: gen coh 1}
Let $V$ be the vector space of $(k+1)$-order tensors of format $\nn=(n_1+1,\dots,n_{k+1}+1)$ with $n_{k+1}>\sum_{i=1}^kn_i$. Given integers $r\geq2$ and $0\leq q\leq r$, then $H^q(\ser)=0$ for $r<k$.
\end{Proposition}
\begin{proof}
By  \Cref{lemma 3.2 DOT} and  \Cref{thm: Kunneth}, in order for the cohomology $H^q(\ser)$ to be nonzero we need, for some choice of $r_i$'s, each $H^{q_i}(\Omega^{r_i}_{\PP^{n_i}})$ to be nonvanishing. We shall analyse each possibility for this cohomology group using the equalities \eqref{eq: coh Omega}. 

Suppose there exists an index $i$ such $q_i=0$ and $2r_i+1-r>r_i$, thus $r_i=r$. This implies that $r_j=0$ for all $j\neq i$ and so all the other cohomology groups are $H^{q_j}(\Omega^0_{\PP^{n_j}}(1-r))$. Since $1-r<0$, we must have $q_j=n_j$. On the other hand, this implies $q=\sum_{i=1}^{k+1,i\neq j}q_i=\sum_{k+1,i\neq j} n_j\geq k$, as $n_j\geq 1$. This is a contradiction showing that all the nonvanishing cases come from the second and third line
of the equalities \eqref{eq: coh Omega}. 

Suppose there exists an index $i$ such that $0\leq q_i = r_i\leq n_i$, as in the second line of the equalities \eqref{eq: coh Omega}. Then 
$2r_i + 1 - r = 0$, which means $r_i=\frac{r-1}{2}\geq 1$. Thus $q_j\geq 1$ for all $1\leq j\leq k+1$. Therefore $q=\sum_{j=1}^{k+1} q_j\geq k+1$, a contradiction.

The last case is when for all $1\leq i\leq k+1$, one has $q_i=n_i$ as in the third line of the equalities \eqref{eq: coh Omega}. Then $q= \sum_{i=1}^{k+1} q_i = \sum_{i=1}^{k+1} n_i>k+1>r$, a contradiction.  
\end{proof}
\begin{Proposition}\label{lemma: gen coh 2}
Let $V$ be the vector space of $(k+1)$-order tensors of format $\nn=(n_1+1,\dots,n_{k+1}+1)$ with $n_{k+1}>\sum_{i=1}^kn_i$. Given integers $r\geq2$ and $0\leq q\leq r$, assuming $k\leq r\leq \sum_{i=1}^kn_i-1$ then:
\begin{enumerate}
    \item[(i)]  for $k+1\geq 4$ it holds that $H^q(\ser)=0$; 
 \item[(ii)] for $k+1=3$, $r> 3$ and odd, one has the nonvanishing cohomology group $H^r(\ser)=H^1(\Omega^1_{\PP^1}(3-r))\otimes H^{(r-1)/2}\left(\Omega^{(r-1)/2}_{\PP^{n_2}}\right)\otimes H^{(r-1)/2}\left(\Omega^{(r-1)/2}_{\PP^{n_3}}\right)$. If $r=3$, also $H^3(\sE^{(3)})=H^1(\Omega^{1}_{\PP^{n_1}})\otimes H^1(\Omega^{1}_{\PP^{n_2}})\otimes H^1(\Omega^{1}_{\PP^{n_3}})$ is nonvanishing.
\end{enumerate}
\end{Proposition}
\begin{proof}

Suppose there exists an index $i$ such that $q_i=0$. Thus $r_i=r$ and $r_j=0$ for all $j\neq i$. Since $r_j=0$, for a nonvanishing cohomology group
we need $q_j=n_j$ for all $j\neq i$. Thus $q=\sum^{k+1}_{j=1, j\neq i} q_j= \sum^{k+1}_{j=1, j\neq i} n_j\geq \sum^{k+1}_{j=1, j\neq i} n_j>r$, where the last
inequality is by assumption. 

As in the proof of \Cref{lemma: gen coh 1}, the case when for all $1\leq i\leq k+1$, one has $q_i=n_i$, cannot appear. Thus there exists an index 
$i$ as in the second line of the equalities \eqref{eq: coh Omega}. For such an index, $q_i=r_i=\frac{r-1}{2}\geq 1$. Hence $r$ must be odd in this case. 

Let $1\leq \ell\leq k+1$ be the number of such indices. Hence $\ell\cdot (r-1)/2\leq r$ with $r\geq 3$. Therefore 
\[
\ell \leq \frac{2r-2+2}{r-1} = 2 + \frac{2}{r-1}\leq 3. 
\] 
Note that $\ell = 3$ if and only if $2/(r-1)=1$ if and only if $r=3$, and $\ell \leq 2$ if and only if $r>3$.

When $r=3$, since there are $\ell=3$ indices such that $q_i\neq 0$ and $\sum_{i=1}^{k+1} q_i \leq r = 3$, we have $k+1=3$ and so the format must be $\nn=(n_1+1,n_2+1,n_3+1)$. Thus the cohomology group we find is $H^3(\sE^{(3)})=H^1(\Omega^{1}_{\PP^{n_1}})\otimes H^1(\Omega^{1}_{\PP^{n_2}})\otimes H^1(\Omega^{1}_{\PP^{n_3}})$. 

Now, we have two cases: either $\ell = 1$ or $\ell=2$. If $\ell=1$, then $q_j=n_j$ for all $j\neq i$. Thus $q=\sum_{j=1}^{k+1} q_j=\sum^{k+1}_{j=1,j\neq i}n_j+r_i\geq \sum_{j=1}^kn_k+1>r$, a contradiction. 

If $\ell=2$, then let $i, j$ be two such indices with $r_i=(r-1)/2=r_j$. This implies that $q=\sum_{s=1}^{k+1} q_s= r-1+\sum_{s\neq i,j} n_s\geq r$. Moreover, equality holds only when $k=2$ and $n_s=1$ for $s\neq i,j$. In such a case, upon renaming the indices, we have $q_3 = 1, q_1 = q_2 = (r-1)/2$.
This case gives either $r_3=0$ and $q_3=n_3=1$, or $1=q_3=r_3\leq n_3$. In the first case, the cohomology vanishes. In the second case, the cohomology does not vanish if and only if the cohomology group is 
\[
H^r(\ser)=H^1(\Omega_{\PP^1}(3-r))\otimes H^{(r-1)/2}(\Omega^{(r-1)/2}_{\PP^{n_2}})\otimes H^{(r-1)/2}(\Omega^{(r-1)/2}_{\PP^{n_3}}),
\]
when $r>3$ and odd. If $r=3$, the cohomology group does not vanish if and only if it is 
\[
H^3(\sE^{(3)})=H^1(\Omega^{1}_{\PP^{n_1}})\otimes H^1(\Omega^{1}_{\PP^{n_2}})\otimes H^1(\Omega^{1}_{\PP^{n_3}}).
\]
This concludes the proof. 
\end{proof}
\begin{Proposition}\label{lemma: gen coh 3}
Let $V$ be the vector space of $(k+1)$-order tensors of format $\nn=(n_1+1,\dots,n_{k+1}+1)$ with $n_{k+1}>\sum_{i=1}^kn_i$. Given nonnegative integers $r\geq2$, $q\leq r$ and $ r\geq \sum_{i=1}^kn_i=n$, then:
\begin{enumerate}
    \item[(i)] for $q<r$ and $q\neq n$ it holds $H^q(\ser)=0$;
    \item[(ii)] for $q=n$ we have that $H^n(\ser)=\left(\bigotimes_{j=1}^kH^{n_j}(\Omega^{0}_{\PP^{n_j}}(1-r))\right)\otimes H^{0}(\Omega^{r}_{\PP^{n_{k+1}}}(r+1))$. Moreover, if  $r=n$, this vanishes if $k=2$ and $n_i=1$ for some $i=1,2$;
    \item[(iii)]for $k+1=3$, $r> 3$ and odd, one has the nonvanishing cohomology group $H^r(\ser)=H^1(\Omega^1_{\PP^1}(3-r))\otimes H^{(r-1)/2}\left(\Omega^{(r-1)/2}_{\PP^{n_2}}\right)\otimes H^{(r-1)/2}\left(\Omega^{(r-1)/2}_{\PP^{n_3}}\right)$. If $r=3$, also $H^3(\sE^{(3)})=H^1(\Omega^{1}_{\PP^{n_1}})\otimes H^1(\Omega^{1}_{\PP^{n_2}})\otimes H^1(\Omega^{1}_{\PP^{n_3}})$ is nonvanishing;
    
    \item[(iv)] for $r\geq 3$ and odd, for $\sum_{i=1}^k n_i+\frac{r-1}{2}=r$, one has the nonvanishing cohomology group $H^r(\ser)=\left(\bigotimes^{k}_{i=1}H^{n_i}(\Omega^{n_i}_{\PP^{n_i}}(2n_i-r+1))\right)\otimes H^{(r-1)/2}(\Omega^{(r-1)/2}_{\PP^{n_{k+1}}})$. 
 
\end{enumerate}

\end{Proposition}
\begin{proof}

Suppose there exists some index $i$ such that $q_i=0$. 
Since we must have $2r_i+1-r>r_i$, we find $r_i=r$.  Thus $r_j=0$ for all $j\neq i$. 

Since $\sum_{\ell=1}^{k+1} r_{\ell}=r$ and $r\geq \sum_{\ell=1}^kn_{\ell}$, then $r> n_{\ell}$ for all $1\leq \ell\leq k$. Thus $i=k+1$. 
Then we find cohomology groups of the form 
$H^{q_j}(\Omega^{0}_{\PP^{n_j}}(1-r))$ for all $j\neq k+1$. Then they are nonzero if and only if $q_j=n_j$ for all $j\neq k+1$. Therefore the only nonzero cohomology group is $H^{n}(\ser)=\left(\bigotimes_{j=1}^kH^{n_j}(\Omega^{0}_{\PP^{n_j}}(1-r))\right)\otimes H^{0}(\Omega^{r}_{\PP^{n_{k+1}}}(r+1))$, which is (ii). 

Suppose that $q_i = n_i$ for all $1\leq i\leq k+1$, i.e. for every index the factor comes from the third line of equalities \eqref{eq: coh Omega}. Since $q=\sum_{i=1}^{k+1}q_i\leq r$ by assumption and $n_i=q_i\leq r_i\leq n_i$, then $q=r=\sum_{i=1}^{k+1}n_i$. In such a case, in the $(k+1)$-th factor we have the vector bundle $\Omega^{n_{k+1}}_{\PP^{n_{k+1}}}(2n_{k+1}-r+1)$. The latter has nonvanishing cohomology only if $2n_{k+1}-r+1<n_{k+1}-n_{k+1}=0$. This implies $n_{k+1}+1-n<0$, a
contradiction. 

Therefore there exists an index 
$i$ satisfying in the second line of the equalities \eqref{eq: coh Omega}. Let $s\geq 1$ be the number of such indices. As in \Cref{lemma: gen coh 2} we find $s\leq 3$, $s = 3$ if and only if $r=3$, and $s\leq 2$ otherwise. The cases $s=2$ and $s=3$ give (iii). 

Suppose $s=1$. Then all other indices satisfy the third line of the equalities \eqref{eq: coh Omega}. This means $q_j=n_j\geq r_j$ for all $j\neq i$ and $q_i=(r-1)/2$.

Note that the condition $2r_{k+1}+1-r<r_{k+1}-n_{k+1}$ implies $n_{k+1}<\sum^k_{\ell =1}r_{\ell}-1$ and so $n_{k+1}<\sum_{\ell=1}^kr_{\ell}\leq n$, a contradiction. Then the index $k+1$ cannot satisfy the third line of equalities \eqref{eq: coh Omega}. Hence $i=k+1$.

At this point, we have  $\sum_{\ell=1}^{k+1}r_{\ell}=r=\sum_{\ell=1}^kr_{\ell}+(r-1)/2$ and $\sum_{\ell=1}^{k+1}q_{\ell}=q=\sum_{\ell=1}^kn_{\ell}+(r-1)/2$. Since $q\leq r$ and $n_{\ell}\geq r_{\ell}$ we find $r_{\ell}=n_{\ell}$. This implies that the only nonvanishing cohomology group appears for $q=r$ odd and $H^r(\ser)=\left(\bigotimes_{i=1}^kH^{n_i}(\Omega_{\PP^{n_i}}^{n_i}(2n_i+1-r)) \right)\otimes H^{(r-1)/2}(\Omega^{(r-1)/2}_{\PP^{n_{k+1}}})$, such that $\sum_{i=1}^k n_i+(r-1)/2=r$. This is the case (iv). 
\end{proof}

The previous lemmas implications are summarized in the next corollary.

\begin{Corollary}\label{coh: chains}
Let $V$ be the vector space of $(k+1)$-order tensors of format $\nn=(n_1+1,\dots,n_{k+1}+1)$ with $n_{k+1}>\sum_{i=1}^kn_i$ and $k+1\geq 4$, let $n=\sum_{i=1}^k n_i$. Then the following chains of isomorphisms and inclusions hold:
\begin{enumerate}
\item[(i)] $H^0(\sF_2({\bf 1}))\cong\cdots\cong H^{n-2}(\sF_{n}({\bf 1}))$; 

\item[(ii)] $H^1(\sF_2({\bf 1}))\cong\cdots\cong H^{n-1}(\sF_n({\bf 1}))\subset H^{n}(\sF_{n+1}({\bf 1}))$; 

\item[(iii)] $H^{n+1}(\sF_{n+3}({\bf 1}))\cong\cdots\cong H^{N-1}(\sF_{N+1}({\bf 1}))=0$, 
where $N=\sum_{i=1}^{k+1} n_i$. 

\item[(iv)]
$H^{n+2}(\sF_{n+2}({\bf 1}))\subset\cdots\subset H^{N}(\sF_{N+1}({\bf 1}))=0$, 
where $N=\sum_{i=1}^{k+1} n_i$. 
\end{enumerate}
\begin{proof}
By \Cref{lemma: gen coh 1}, $H^{r-2}(\sE^{(r)}) = H^{r-1}(\sE^{(r)}) = 0$
for $2\leq r\leq k-1$. Thus the exact sequence \eqref{longcoh} implies (i) and (ii), where the last inclusion is to be understood as an injective map induced by the exact sequence. Similarly, \Cref{lemma: gen coh 2} for $k+1\geq 4$ implies (iii). \Cref{lemma: gen coh 3} implies that $H^r(\sE^{(r+1)})=0$ for $r\geq n+1$, then the map $H^r(\sF_{r+1}({\bf 1}))\to H^{r+1}(\sF_{r+2}({\bf 1}))$ is injective.
\end{proof}
\end{Corollary}

\section{A map in cohomology giving the codimension of $\langle Z_T\rangle$ in $\PP(H_T)$: Case $k+1\geq 4$ factors}

Using cohomological maps from the sequence \eqref{eq: Koszul}, we shall find the information about the codimension of $\langle Z_T\rangle$ in $\PP(H_T)$. 
This will be encoded in a map $\alpha_T:H^n(\sE^{(n+1)})\to H^n(\sE^{(n)})$ induced by the Koszul map.

\begin{Proposition}\label{prop:dim k+1-factors}
Let $n=\sum_{j=1}^kn_j$ and let $V$ be a $(k+1)$-tensor space $V$ of format $(n_1+1,\dots,n_k+1,n+2)$ with $k\geq 3$.  Let $T\in V$ be a general tensor.
Then there exists an induced map $\alpha_T:H^n(\sE^{(n+1)})\to H^n(\sE^{(n)})$ such that the codimension of $\langle Z_T\rangle$ inside $\PP(H_T)$ is the given by the dimension of the kernel of $\alpha_T$. 
\begin{proof}
From \Cref{lemma: gen coh 1}, \Cref{lemma: gen coh 2}, \Cref{lemma: gen coh 3} we have that 
\begin{align*}
    H^0(\sF_2({\bf 1}))\cong\dots\cong H^{2n-1}(\sF_{2n+1}({\bf 1}))=0;\\
    H^1(\sF_2({\bf 1}))\cong\dots\cong H^{n-1}(\sF_{n}({\bf 1}))\subset H^n(\sF_{n+1}({\bf 1}));\\
    H^{n+1}(\sF_{n+2}({\bf 1}))\subset\dots\subset H^n(\sF_{2n+1}({\bf 1}))= 0.
\end{align*}
We use the long exact sequence of cohomologies from \eqref{eq: shortexactsequence}. The sequence 

\begin{align*}
    \dots\to H^{n-1}(\sE^{(n+1)})\to H^{n-1}(\sF_{n+1}({\bf 1}))\to H^n(\sF_{n+2}({\bf 1}))\to \\
    \to H^n(\sE^{(n+1)})\to H^{n}(\sF_{n+1}({\bf 1}))\to H^{n+1}(\sF_{n+2}({\bf 1}))\to\dots
\end{align*}

gives that $H^{n-1}(\sF_{n+1}({\bf 1}))=0$ and $H^n(\sF_{n+1}({\bf 1}))\cong H^n(\sE^{(n+1)})\cong\left(\bigotimes_{j=1}^kH^{n_j}(\Omega_{n_j}^0(-n))\right)\otimes H^0\left(\Omega_{n+2}^{n+1}(n+2)\right)$. The next step is to look at the sequence 

\begin{align*}
    \dots\to H^{n-2}(\sE^{(n)})\to H^{n-2}(\sF_{n}({\bf 1}))\to H^{n-1}(\sF_{n+1}({\bf 1}))\to H^{n}(\sF_{n+1}({\bf 1}))\to\\
    \to H^{n-1}(\sE^{(n)})\to H^{n-1}(\sF_{n}({\bf 1}))\to H^{n}(\sF_{n+1}({\bf 1}))\xrightarrow{\beta}H^n(\sE^{(n)})\to\dots
\end{align*}
It implies that $H^{n-2}(\sF_n({\bf 1}))=0$, moreover we have from $$0\to H^{n-1}(\sF_{n}({\bf 1}))\to H^{n}(\sF_{n+1}({\bf 1}))\xrightarrow{\alpha}H^n(\sE^{(n)})$$ that $h^{n-1}(\sF_{n}({\bf 1}))=\dim(\ker(\beta))$, moreover $H^n(\sE^{(n)})\cong \left(\bigotimes_{j=1}^kH^{n_j}(\Omega^0_{n_j}(-n+1))\right)\otimes H^0\left(\Omega^n_{n+1}(n+1)\right)$.

Then consider the map $\alpha_T$ defined as the composition of the maps in the diagram:
\[
\begin{tikzcd}
H^n(\sE^{(n+1)}) \arrow[rr, "\alpha_T"] \arrow[rd, "\cong"] & & H^n(\sE^{(n)}) \\
& H^n(\sF_{n+1}({\bf 1})) \arrow[ru, "\beta"] &               
\end{tikzcd}
\]
Note that $\dim(\ker(\beta)) = \dim(\ker(\alpha_T))$. Recall that $H^{n-1}(\sF_n({\bf 1}))\cong H^1(\sF_2({\bf 1}))$ and that we have 
the sequence
\[
0\to \sF_{2}({\bf 1})\to \sE^{(1)}\to \sI_{Z_T}({\bf 1})\to 0\,
\]
which induces 
\[
0\rightarrow H^0(\sE^{(1)})\rightarrow H^0(\mathcal I_Z({\bf 1}))\rightarrow H^1(\sF_2({\bf 1}))\rightarrow 0. 
\]
By \Cref{lem: interpretation of H^0(E^1)}, we see that $h^1(\sF_2({\bf 1})) = h^{n-1}(\sF_n({\bf 1}))$ is the codimension of $\langle Z_T\rangle$ in $\PP(H_T)$. 
The latter was shown to be equal to $\dim(\ker(\alpha_T))$. 
\end{proof}
\end{Proposition}

\subsection{Spelling out the map in cohomology}
In the previous section, we showed that the codimension of $\langle Z_T\rangle$ in $H_T$ is provided by the kernel of the map $\alpha_T:H^n(\sE^{(n+1)})\to H^n(\sE^{(n)}$. In order to better understand the structure we aim to make $\alpha_T$ more explicit in terms of representation theory. 

To do so we need to use Bott's Theorem and the theory of weights. We also remark that the analysis is valid and useful later for the three-factors case in \Cref{sec: k=2} when $n_1,n_2\geq2$, thus in this subsection we assume  $k+1\geq 3$. The case $n_1=1$ is special and is treated separately in Theorem \ref{Thm: (2,n,n+2)}.

\begin{Lemma}
We have that 
\[
\alpha_T:\bigotimes_{j=1}^k \sS^{n-n_j-1}\C^{n_j+1}\longrightarrow \left(\bigotimes_{j=1}^k\sS^{n-n_j-2}\C^{n_j+1}\right)\otimes\C^{n+2}.
\]
\begin{proof}
Recall that 
\[
H^n(\sE^{(n+1)})\cong\left(\bigotimes_{j=1}^kH^{n_j}(\Omega_{n_j}^0(-n))\right)\otimes H^0\left(\Omega_{n+2}^{n+1}(n+2)\right)
\] 
and
\[
H^n(\sE^{(n)})\cong \left(\bigotimes_{j=1}^kH^{n_j}(\Omega^0_{n_j}(-n+1))\right)\otimes H^0\left(\Omega^n_{n+1}(n+1)\right).
\] 
The weights associated to $\Omega^{n+1}_{n+1}(n+2)$ and $\Omega^{n}_{n+1}(n+1)$ are respectively $\lambda_{n+2}$ and $\lambda_{n+1}$, thus those spaces are isomorphic respectively to $\bigwedge^{n+2}\C^{n+2}\cong\C$ and $\bigwedge^{n+1}\C^{n+2}\cong\C^{n+2}$. On the other hand, $\Omega^0_{n_j}(-t)=\sO_{n_j}(-t)$, thus by Serre's duality we have that $H^{n_j}(\sO_{n_j}(-t))=H^0(\sO(t-n_j-1))\cong \sS^{t-n_j-1}\C^{n_j+1}$.
\end{proof}
\end{Lemma}

Now to derive the actual definition of $\alpha_T$, we need to take a step back and define an $\SL(n_1+1)\times\dots\times\SL(n_1+1)\times \SL(n+2)$-equivariant map $\alpha$ describing the functorial change of the Koszul map $\alpha_T$ as $T$ varies.

We first show the uniqueness of the $\SL(n_1+1)\times\dots\times\SL(n_1+1)\times \SL(n+2)$-equivariant map up to identifications given by Bott's Theorem. 

\begin{Lemma}
Suppose $k\geq3$ or $n_1,n_2\geq2$ otherwise. There is a unique, up to scalar multiplication, non-zero  $\SL(n_1+1)\times\dots\times \SL(n_k+1)\times \SL(n+2)$-equivariant map $$\bigotimes_{j=1}^k \left(\sS^{n-n_j-1}\C^{n_j+1}\right)\otimes(\C^{n_1+1}\otimes\dots\otimes \C^{n+2})^\ast\to \left(\bigotimes_{j=1}^k\sS^{n-n_j-2}\C^{n_j+1}\right)\otimes\C^{n+2},$$
    where $n=\sum n_i$.
\end{Lemma}
\begin{proof}
    The domain can be decomposed by irreducible $\SL(n_1+1)\times\dots\times \SL(n_k+1)\times \SL(n+2)$-representations as $$\left(
    \bigotimes_{i=1}^k \left(\sS^{n-n_i}\C^{n_i+1}\oplus \sS^{n-n_i-2}\C^{n_i+1} \right)\right)\otimes (\C^{n+2})^\ast,
    $$
    by Schur's Lemma the map is an isomorphism from $\left(\bigotimes_{i=1}^k \sS^{n-n_i-2}\C^{n_i+1} \right)\otimes (\C^{n+2})^\ast$ to itself, given by scalar multiplication, and zero everywhere else.
\end{proof}

\begin{Lemma}\label{lemma: spelloutalpha}
    The map 
\[
    \varphi:\bigotimes_{j=1}^k \left(\sS^{n-n_j-1}\C^{n_j+1}\right)\otimes(\C^{n_1+1}\otimes\dots\otimes \C^{n+2})^\ast\to \left(\bigotimes_{j=1}^k\sS^{n-n_j-2}\C^{n_j+1}\right)\otimes\C^{n+2}
\]
    defined by 
 \[
 \varphi(f_1\otimes\dots \otimes f_k\otimes T) = \sum_{i_1,\dots,i_k=1}^{n_1+1,\dots,n_k+1}\left(\otimes_{j=1}^k \partial_{i_j}f_j\right)\otimes T_{i_1\dots i_k},
\]
where $T=\sum_{i_1\dots i_k} e_{i_1}\otimes\dots\otimes e_{i_k}\otimes T_{i_1\dots i_k}$ with $T_{i_1\dots i_k}\in\C^{n+2}$, is $\SL(n_1+1)\times\dots\times \SL(n+2)$-equivariant.
\end{Lemma}
\begin{proof}
    Let $g=g_1\times\dots\times g_{k+1}\in \SL(n_1+1)\times\dots\times \SL(n+2)$, and $\xx_{i}=(x_{i_1},\dots,x_{i_{n_i+1}})$ with $x_{i_j}^\ast=e_{i_j}$, and the action of $g_i$ on $\C^{n_i+1}$ is the natural action $g_i\cdot e_{i_j}=e_{i_j}(g_i^{-1})$. Then 
    \begin{align*}
        \varphi(g\cdot f_1(\xx_1)\otimes\dots\otimes f_k(\xx_k)\otimes T)&=\sum_{i_1,\dots,i_k=1}^{n_1+1,\dots,n_k+1}\left(\otimes_{j=1}^k (g\cdot \partial_{i_j})f_j(g_j^{-1}\cdot\xx_j)\right)\otimes g_{k+1}\cdot T_{i_1\dots i_k}\\ &=g\cdot\sum_{i_1,\dots,i_k=1}^{n_1+1,\dots,n_k+1}\left(\otimes_{j=1}^k  \partial_{i_j}f_j(\xx_j)\right)\otimes  T_{i_1\dots i_k}\\&=g\cdot\varphi(f_1\otimes\dots\otimes f_k\otimes T).\end{align*}
This shows the claim.
\end{proof}

\begin{Corollary}\label{cor:spelloutalpha}
The map $\alpha_T$ is then $\varphi(\cdot \otimes T)$ defined by 
\[
\alpha_T(f_1\otimes\dots \otimes f_k) = \sum_{i_1,\dots,i_k=1}^{n_1+1,\dots,n_k+1}\left(\otimes_{j=1}^k \partial_{i_j}f_j\right)\otimes T_{i_1\dots i_k}, 
\]
where $T=\sum_{i_1\dots i_k} e_{i_1}\otimes\dots\otimes e_{i_k}\otimes T_{i_1\dots i_k}$ with $T_{i_1\dots i_k}\in\C^{n+2}$. 
\end{Corollary}

\section{A map in cohomology giving the codimension of $\langle Z_T\rangle$ in $\PP(H_T)$: Case $k+1=3$ factors}\label{sec: k=2}

In this section, we discuss the three factors that present some differences from the general behavior. Indeed,  we started by noticing the difference in the chain of isomorphisms and containment at the beginning of the proof of \Cref{prop:dim k+1-factors}. We first look at $V=\C^2\otimes \C^{n}\otimes \C^{n+2}$.

\begin{Corollary}\label{cor:chains k=3}
    If $k+1=3$ and $n_1=1$, then \Cref{coh: chains} can be refined to: \begin{enumerate}
    \item[(i)] $H^0(\sF_2({\bf 1}))\cong\dots\cong H^{n-2}(\sF_n({\bf 1}))\subset H^{n-1}(\sF_{n+1}({\bf 1}))\cong\dots= 0$.
    \item[(ii)] $H^1(\sF_1({\bf 1}))\cong H^2(\sF_3({\bf 1}))\subset H^3(\sF_4({\bf 1}))\cong \dots\cong H^{n-1}(\sF_{n}({\bf 1}))\subset H^n(\sF_{n+1}({\bf 1}))$, if $n$ is odd.
    \item[(iii)] $H^1(\sF_1({\bf 1}))\cong H^2(\sF_3({\bf 1}))\subset H^3(\sF_4({\bf 1}))\cong \dots\subset H^{n-1}(\sF_{n}({\bf 1})) \cong H^n(\sF_{n+1}({\bf 1}))$, if $n$ is even.
    \item[(iv)] $H^{n+1}(\sF_{n+2}({\bf 1}))=\cdots = 0$.
\end{enumerate}
\end{Corollary}

\mainthm
\begin{proof}
    The cases $n=2,3$ are solved in \cite{ST22}. So we assume $n\geq4$. From the exact sequence $$0\to \sF_{r+1}({\bf 1})\to \ser\to \sF_r({\bf 1})\to 0,$$ we obtain the long exact sequence 
\begin{align*}
    \dots\to H^{r-2}(\ser)\to H^{r-2}(\sF_r({\bf 1}))\to H^{r-1}(\sF_{r+1}({\bf 1}))\to H^{r-1}(\ser)\to \\ \to H^{r-1}(\sF_r({\bf 1})) \to H^r(\sF_{r+1}({\bf 1}))\to H^r(\ser) \to H^r(\sF_r({\bf 1}))\to H^{r+1}(\sF_{r+1}({\bf 1}))\to 0
\end{align*}
For $r=n+1$, we derive the exact sequence
\begin{align*}
    0\to H^n(\sE^{(n+1)})\to H^n(\sF_{n+1}({\bf 1}))\to H^{n+1}(\sF_{n+2}({\bf 1}))=0.
\end{align*}
Thus $H^n(\sE^{(n+1)})\cong H^n(\sF_{n+1}({\bf 1}))$. Moreover $h^n(\sE^{(n+1)})=n-1$, by \Cref{lemma: gen coh 3}(ii).

Assume $r=n$ odd. Then  
\begin{align*}
0\to H^{n-1}(\sF_n({\bf 1}))\to H^{n}(\sF_{n+1}({\bf 1}))\xrightarrow{\beta} H^n(\sE^{(n)})\to \dots,
\end{align*}
where $H^n(\sE^{(n)})\neq 0$. Therefore, we have that $h^{n-1}(\sF_n({\bf 1}))=\dim (\ker(\beta))$. Moreover, from the following diagram 
\[
\begin{tikzcd}
H^n(\sE^{(n+1)}) \arrow[rr, "\alpha_T"] \arrow[rd, "\cong"] & & H^n(\sE^{(n)}) \\
& H^n(\sF_{n+1}({\bf 1})) \arrow[ru, "\beta"] &               
\end{tikzcd}
\]
one has $\dim(\ker(\alpha_T)) = \dim(\ker(\beta))$. By \Cref{lemma: gen coh 2}, \Cref{lemma: gen coh 3} and Bott's {Theorem} \ref{bott form} we have
\begin{enumerate}
    \item $H^3(\sE^{(3)})=H^1(\Omega_{\PP^1}^1)\otimes H^1(\Omega_{\PP^{n-1}}^1) \otimes H^1(\Omega_{\PP^{n+1}}^1)\cong \C$.
    \item $H^n(\sE^{(n+1)})= H^{1}(\Omega^0_{\PP^1}(-n))\otimes H^{n-1}(\Omega^0_{\PP^{n-1}}(-n))\otimes H^0(\Omega^{n+1}_{\PP^{n+1}}(n+2))\cong \mathcal S^{n-2}\C^2\otimes \C\otimes\C$.
    \item If $r> 3$ and odd, one has the nonvanishing cohomology group $H^r(\ser)=H^1(\Omega^1_{\PP^1}(3-r))\otimes H^{(r-1)/2}\left(\Omega^{(r-1)/2}_{\PP^{n_2}}\right)\otimes H^{(r-1)/2}\left(\Omega^{(r-1)/2}_{\PP^{n_3}}\right)\cong S^{n-3}\C^2\otimes \C\otimes\C$.
\end{enumerate}

The map $\alpha_T$ is not $\mathrm{SL}(2)\times\mathrm{SL}(n)\times \mathrm{SL}(n+2) $-equivariant since it depends on $T$, however the map $$\alpha:\mathcal S^{n-2}\C^2\otimes(\C^2\otimes \C^n\otimes\C^{n+2})^\ast\to \mathcal S^{n-3}\C^2,\ (f,T)\mapsto \alpha_T(f)$$ is $\mathrm{SL}(2)\times\mathrm{SL}(n)\times \mathrm{SL}(n+2) $-equivariant. 

We will show that the null map is the only $\mathrm{SL}(2)\times\mathrm{SL}(n)\times \mathrm{SL}(n+2)$-equivariant map in such case. Indeed the domain decomposes as an $\mathrm{SL}(2)\times\mathrm{SL}(n)\times \mathrm{SL}(n+2)$-irreducible representation as 
\[
\left(S^{n-1}\C^2\oplus \mathcal S^{n-3}\C^2\right) \otimes (\C^n)^\ast\otimes (\C^{n+2})^\ast.
\]
By Schur's Lemma, there is only the null map from the domain to $\mathcal S^{n-3}\C^2\otimes \C\otimes \C$. This implies that also $\alpha_T$ is null.

Assume by induction that the same holds for any $r+2\leq n$ odd, for $r\geq 3$ we have the sequence

\[
0\to H^{r-1}(\sF_{r}({\bf 1}))\to H^{r}(\sF_{r+1}({\bf 1}))\xrightarrow{\alpha} H^{r}(\sE^{(r)})\to \dots
\]

From the induction hypothesis $H^{r}(\sF_{r+1}({\bf 1}))\cong \mathcal S^{n-2}\C^2$, and $H^{r}(\sE^{(r)})\cong \mathcal  S^{r-3}\C^2$. Similar as above, by tensoring the domain by $(\C^2\otimes \C^n\otimes \C^{n+2})^\ast$, the only $\mathrm{SL}(2)\times\mathrm{SL}(n)\times \mathrm{SL}(n+2)$-equivariant map is the trivial map, thus $\alpha=0$. 

We conclude that $H^{1}(\sF_2({\bf 1}))\cong H^{n-1}(\sF_n({\bf 1}))\cong \mathcal S^{n-2}\C^2$. Utilizing the exact sequence $$
0\to \sF_2({\bf 1})\to H^0(\sE^{(1)}\to \sI_{Z_T}(1)\to 0,
$$
taking the long exact sequence in cohomology we have $$
0\to H^0(\sE^{(1)})\to H^0(\sI_{Z_T}(1))\to H^1(\sF_2({\bf 1}))\to 0,
$$
and $\dim \sI_{Z(T)}(1)=1+\binom{n}{2}+\binom{n+2}{2}+n-1$. The first three summands correspond to the codimension of $H_T$, thus $n-1$ is the codimension of $\langle Z_T\rangle \subset H_T$.

If $r=n$ is even, then $H^{n-1}(\sF_n({\bf 1}))\cong H^n(\sF_{n+1}({\bf 1}))$. Then for $r=n-1$ the sequence becomes $$
0\to H^{n-2}(\sF_{n-1}({\bf 1}))\to H^{n-1}(\sF_{n}({\bf 1}))\xrightarrow{\beta} H^{n-1}(\sE^{(n-1)})\to \dots
$$
The map $\beta$, via the isomorphism, can be understood as $H^{n}(\sE^{(n+1)})\to H^{n-1}(\sE^{(n-1)})$. So it is a map $\mathcal S^{n-2}\C^2\to \mathcal S^{n-4}\C^2$. With the  same technique as the previous case, we deduce that $\langle Z_T\rangle$ has codimension $n-1$ in $H_T$.    
\end{proof}

We now address the case $V=\C^{a+1}\otimes \C^{b+1}\otimes \C^{a+b+2}$ with $a,b\geq 2$.

\begin{Corollary}\label{cor:chains k=3 gen}
    If $k+1=3$ and $n_i\geq 2$, then \Cref{coh: chains} can be refined to: \begin{enumerate}
    \item[(i)] $H^0(\sF_2({\bf 1}))\cong\dots\cong H^{n-2}(\sF_n({\bf 1}))\subset H^{n-1}(\sF_{n+1}({\bf 1}))\cong\dots= 0$.
    \item[(ii)] $H^1(\sF_1({\bf 1}))\cong H^2(\sF_3({\bf 1}))\subset H^3(\sF_4({\bf 1}))\cong \dots\cong H^{n-1}(\sF_{n}({\bf 1}))\subset H^n(\sF_{n+1}({\bf 1}))$.
    \item[(iii)] $H^{n+1}(\sF_{n+2}({\bf 1}))=\cdots = 0$.
\end{enumerate}
\end{Corollary}

\begin{Proposition}\label{prop: dim k=3 alpha}
    Let $V=\C^{a+1}\otimes \C^{b+1}\otimes \C^{n+2}$ where $a,b\geq2$ and $n=a+b$. Let $T\in V$ be a general tensor. Then there exist induced maps $\alpha_T: H^{n}(\sE^{(n+1)})\to H^{n}(\sE^{(n)})$ such that the codimension of $\langle Z_T\rangle $ inside $\PP(H_T)$ is given by $\dim(\ker(\alpha_T))$.
\end{Proposition}
\begin{proof}

By \Cref{cor:chains k=3 gen}, we have to understand $h^2(\sF_3)$ to understand the dimension of $\langle Z_T\rangle$. In order to do that we first study $H^{n-1}(\sF_n)$.

From the exact sequence $$0\to \sF_{r+1}\to \ser\to \sF_r\to 0,$$ we obtain the long exact sequence of cohomologies:
\begin{align*}\label{eq: longexactseq}
    \dots\to H^{r-2}(\ser)\to H^{r-2}(\sF_r({\bf 1}))\to H^{r-1}(\sF_{r+1}({\bf 1}))\to H^{r-1}(\ser)\to \\ \to H^{r-1}(\sF_r({\bf 1})) \to H^r(\sF_{r+1}({\bf 1}))\to H^r(\ser) \to H^r(\sF_r({\bf 1}))\to H^{r+1}(\sF_{r+1}({\bf 1}))\to 0
\end{align*}

For $r=n+1$ the sequence becomes: 
\begin{align*}
    0\to H^n(\sE^{(n+1)})\to H^n(\sF_{n+1}({\bf 1}))\to H^{n+1}(\sF_{n+2}({\bf 1}))=0.
\end{align*}
We obtain $$H^n(\sF_{n+1}({\bf 1}))\cong H^n(\sE^{(n+1)})\cong H^n(\sE^{(n+1)})=H^a(\Omega^0_{\PP^a}(-n))\otimes H^a(\Omega^0_{\PP^a}(-n))\otimes H^0(\Omega^{n+1}_{\PP^{n+1}}(n+2)).$$

For $r=n$ we have 
\begin{align*}
0\to H^{n-1}(\sF_n({\bf 1}))\to H^{n}(\sF_{n+1}({\bf 1}))\xrightarrow{\beta} H^n(\sE^{(n)})\to \dots
\end{align*}
Notice that $H^{n-1}(\sF_n({\bf 1}))=\ker \beta$. From the diagram
\[
\begin{tikzcd}
H^n(\sE^{(n+1)}) \arrow[rr, "\alpha_T"] \arrow[rd, "\cong"] & & H^n(\sE^{(n)}) \\
& H^n(\sF_{n+1}({\bf 1})) \arrow[ru, "\beta"] &            
\end{tikzcd}
\]
we have that $h^{n-1}(\sF_n({\bf 1}))=\dim(\ker(\alpha_T)).$

For $r=3$ the sequence becomes
$$
0\to H^2(\sF_3({\bf 1}))\to H^3(\sF_4({\bf 1}))\xrightarrow{\gamma_T} H^{3}(\sE^{(3)})\to \dots
$$
We recall that $H^{3}(\sE^{(3)}))\cong\C$. So $h^1(\sF_2({\bf 1}))=h^2(\sF_3({\bf 1}))$ or $h^1(\sF_2({\bf 1}))=h^2(\sF_3({\bf 1}))-1$. 

We now analyze $\gamma_T$. First, we recall that as in the proof of Theorem \ref{Thm: (2,n,n+2)},  tensoring the domain of $\alpha_T$ by $V^\ast$ induces an $\mathrm{SL}(a+1)\times\mathrm{SL}(b+1)\times\mathrm{SL}(n+2)$-equivariant map $\alpha$. Since $H^{n}(\sE^{(n+1)})\cong \mathcal S^{b-1}\C^{a+1}\otimes\mathcal S^{a-1}\C^{b+1} $, we have that $H^{n}(\sE^{(n+1)})\otimes V^\ast$ representation is \begin{equation}\label{eq: equi3factor}
(\mathcal S^{b}\C^{a+1}\oplus S^{b-2}\C^{a+1})\otimes(\mathcal S^{a}\C^{b+1}\oplus S^{a-2}\C^{b+1})\otimes (\C^{n})^\ast.
\end{equation}
In particular, $H^{n-1}(\sF_n(1))\otimes V^\ast\cong H^3(\sF_4(1))\otimes V^\ast$ is a sub-representation of \eqref{eq: equi3factor}. Again by tensoring the domain of $\gamma_T$ by $V^\ast$, we obtain an equivariant map $\gamma: H^{n-1}(\sF_n(1))\otimes V^\ast\to \C$, since it is a sub-representation of \eqref{eq: equi3factor}, it follows by Schur's Lemma that $\gamma=0$, so $\gamma_T=0$, thus $H^{2}(\sF_3(1))\cong H^3(\sF_4(1))$.

From 
the sequence
\[
0\to \sF_{2}({\bf 1})\to \sE^{(1)}\to \sI_{Z_T}({\bf 1})\to 0\,
\]
we deduce 
\[
0\rightarrow H^0(\sE^{(1)})\rightarrow H^0(\mathcal I_Z({\bf 1}))\rightarrow H^1(\sF_2({\bf 1}))\rightarrow 0. 
\]
By  \Cref{lem: interpretation of H^0(E^1)}, we see that $h^1(\sF_2({\bf 1}))= h^{n-1}(\sF_n)=\dim(\ker(\alpha))$ is the codimension of $\langle Z_T\rangle$ in $\PP(H_T)$. 
\end{proof}
We recall that the map $\alpha_T$ is spelled out explicitly in \Cref{cor:spelloutalpha}.
\section{An inequality of dimensions}\label{sec:inequality}

In this section we discuss the technical aspects to compare the dimensions of the domain and codomain of the map $\alpha_T$. These are 
\[
\prod_{j=1}^k\binom{n-1}{n_j}
\ \ \mbox{ and } \ \  
(n+2)\prod_{j=1}^k\binom{n-2}{n_j},
\]
where $n=\sum_{j=1}^k n_j$.

\subsection{Case $k=2$}\label{ineqk2}
\begin{Lemma}
    Suppose $1\leq a\leq b$. Then $\binom{a+b-1}{b}\binom{a+b-1}{a}\leq \binom{a+b-2}{b}\binom{a+b-2}{a}(a+b+2)$, unless $a=1$ and $b\geq 2$, or $(a,b)\in \{(2,2),(2,3),(2,4)\}$.
\end{Lemma}
\begin{proof}

Suppose first that $b\geq a\geq 4$. The inequality we have to establish is 
\begin{equation}\label{ineq k=2}
(b-1)(a-1)(a+b+2)\geq (a+b-1)^2.
\end{equation}
Since $(a+b+2)>(a+b-1)$, it is enough to check 
$(a-1)(b-1)\geq (a+b-1)$. 
We may write this as 
\[
(a-2)b-2a+2\geq 0,
\]
which is satisfied for $b\geq a\geq 4$.

When $a=1$,
the inequality is only satisfied for $a=b=1$. When $a=2$, 
the inequality becomes
\[
(b-1)(b+4)\geq (b+1)^2,
\]
which is satisfied for $b\geq 5$. If $a=3$, the inequality becomes
\[
2(b-1)(b+5)\geq (b+2)^2, 
\]
which holds whenever $b\geq 4$. Moreover, the original inequality is also satisfied if $b=3$. 
\end{proof}

\begin{Corollary}
    Let $T\in \C^{a+1}\otimes C^{b+1}\otimes \C^{a+b+2}$, assume $b\geq a\geq 2$ then $\langle Z_T\rangle \subsetneq \PP(H_T)$ if $(a,b)\in \{(2,2),(2,3),(2,4))\}$.
\end{Corollary}

\subsection{Case $k\geq 3$}
Suppose first that $n_j=a$ for all $1\leq j\leq k$. Then the inequality to verify becomes 
\begin{equation}\label{ineq n_j=a}
[(k-1)a-1]^{k}(ka+2)\geq (ka-1)^k
\end{equation}
We know already the cases when the inequality holds for $k=2$ and $a\geq 1$. 

\begin{Lemma}\label{all equal factors}
The inequality $[(k-1)a-1]^{k}(ka+2)\geq (ka-1)^k$ holds
for any $k\geq 3$ and $a\geq 2$. 
\begin{proof}
Since $(ka+2)>(ka-1)$ the statement is proven once 
we prove that $[(k-1)a-1]^{k}\geq (ka-1)^{k-1}$. 
For any fixed $k\geq 3$, define 
\[
f(a)= k\log((k-1)a-1)-(k-1)\log(ka-1).
\]
The function $f$ is differentiable and its derivative 
is $f'(a)=\frac{k(k-1)a}{(ka-1)[(k-1)a-1)}$. Note that
$f'(a)\geq 0$ whenever $a\geq 2$, for any fixed $k\geq 3$. 
Hence for any fixed $k$ and for any $a\geq 2$, we have
\[
f(a)\geq f(2)=g(k) = (2k-3)^k - (2k-1)^{k-1}.
\]
We have to show that $g(k)\geq 0$ for any $k\geq 3$, 
to finish the proof. Note that $g(3)\geq 0$. It would be enough to check that $g$ is increasing. Now consider 
the derivative 
\[
g'(k) = \frac{2k}{2k-3}-\frac{2(k-1)}{2k-1}+\log(2k-3)-\log(2k-1).
\]
Thus $g'(k)$ is a differentiable function in the interval $[3,+\infty)$ and we calculate
\[
g''(k) = -\frac{4(4k^2-4k+3)}{(4k^2-8k+3)^2}<0,\ \ \forall k\geq 3.
\]
Thus $g'(k)$ is a decreasing function in $[3,+\infty)$ and we see that $\lim_{h\rightarrow +\infty} g'(h)=0$. This implies that $g'(k)\geq 0$ for any $k\geq 3$, which proves the statement. 
\end{proof}

\end{Lemma}

\begin{Remark}
Suppose $a=1$. Then for $k=3$, the inequality does not hold 
and gives the case $(2,2,2,5)$. For any $k\geq 4$ and $a=1$, one can check that the inequality \eqref{ineq n_j=a} is satisfied. 
\end{Remark}

\begin{Remark}
Thus, whenever $n_j=a\geq 1$ (for every $j$) and $k\geq 3$, the only exception to the inequality \eqref{ineq n_j=a} is when $n_1=n_2=n_3=1$. 
\end{Remark}

\subsection{General case: set-up}\label{subsectiongenineq}

Here we look at $k$-tuples as unordered collections of positive integers. 

Given the $k$-tuple $(n_1,\ldots, n_k)=(a,\ldots, a)$ as in the previous section, an $\ell$-increment of such a $k$-tuple is a $k$-tuple of the form
$(a+t_1,\ldots, a+t_{\ell}, a,\ldots, a)$, where $t_j\geq 1$. 

The $0$-increment of $(a,\ldots, a)$ is the tuple itself and it is clear that any tuple is an increment of some tuple of this form.

Suppose we start from the $\ell$-increment of the $k$-tuple $(a,\ldots, a)$.
Let $S=\sum_{i=1}^\ell t_i$ and $R_j = \sum_{i\neq j, i=1}^\ell t_i$. 

Then the inequality we have to prove is 
\begin{equation}\label{l increment}
\prod_{j=1}^{\ell} ((k-1)a-1+R_j)\cdot ((k-1)a-1+S)^{k-\ell}\geq (ka-1+S)^{k-1}. 
\end{equation}
Note that for $\ell=0$, $k\geq 3$ and $a\geq 2$ or $k\geq 4$ and $a=1$, this is established in Lemma \Cref{all equal factors}. 
The latter case is the base of our strong induction on $\ell$. Clearly, $0\leq \ell\leq k-1$. 
We are done if we prove that the validity of the inequality for the $\ell$-increment implies the validity of the inequality for the $(\ell+1)$-increment.
Thus, assume that
the inequality \eqref{l increment}
holds true and we would like to prove that 
\[
\prod_{j=1}^{\ell} ((k-1)a-1+R_j+t_{\ell+1})\cdot ((k-1)a-1+S+t_{\ell+1})^{k-\ell-1}((k-1)a+S))\geq (ka-1+S+t_{\ell+1})^{k-1}, 
\]
where $t_{\ell+1}\geq 1$. The right-hand side can be written as 
\[
\sum_{m=0}^{k-1} \binom{k-1}{m} (ka-1+S)^{k-1-m}t_{\ell+1}^m.
\]
Now, we look at the coefficient of $t^m_{\ell+1}$
on the left-hand side for $m\leq k-1$. This is given by 
\begin{equation}\label{coeff of LHS}
\sum_{r=0}^m \binom{k-\ell-1}{r}((k-1)a-1+S)^{k-\ell-r}\cdot \left(\sum_{J\subset [\ell], \ \ |J|=\ell-m-r} \prod_{f=1}^{\ell-m+r} ((k-1)a-1+R_{j_f})\right),
\end{equation}
where $J = \lbrace j_1,\ldots, j_{\ell-m+r}\rbrace\subset [\ell]$.
Each 
summand of \eqref{coeff of LHS} is of the form 
\[
\prod_{f=1}^{\ell-m+r} ((k-1)a-1+R_{j_f})\cdot ((k-1)a-1+S)^{k-\ell-r}.
\]
Since $\ell-m+r\leq \ell$, by strong induction, we have 
\[
\prod_{f=1}^{\ell-m+r} ((k-1)a-1+R_{j_f})\cdot ((k-1)a-1+S)^{k-\ell-r+m}\geq (ka-1+S)^{k-1}
\]
Since $(k-1)a-1+S<ka-1+S$, 
we have $\frac{1}{((k-1)a-1+S)^m}\geq\frac{1}{(ka-1+S)^m}$ for $m\geq 0$. Multiplying this inequality with the previous we obtain
\[
\prod_{f=1}^{\ell-m+r} ((k-1)a-1+R_{j_f})\cdot ((k-1)a-1+S)^{k-\ell-r}\geq (ka-1+S)^{k-1-m}.
\]
Thus our original inequality is proven if we prove that 
\[
\sum_{r=0}^m \binom{k-\ell-1}{r}\cdot \binom{\ell}{\ell-(m-r)}(ka-1+S)^{k-1-m} \geq \binom{k-1}{m}(ka-1+S)^{k-1-m}.
\]
Actually, this is an equality. 
\begin{Lemma}
One has
\[
\sum_{r=0}^m \binom{k-\ell-1}{r}\cdot \binom{\ell}{\ell-(m-r)}=\sum_{r=0}^m \binom{k-\ell-1}{r}\cdot \binom{\ell}{m-r} =\binom{k-1}{m},
\]
for any $0\leq r\leq m\leq k-1$ and any $0\leq \ell\leq k-1$. 
\begin{proof}
Let $X$ be a set of cardinality $k-1$
and let $X_m$ be the collection of all subsets of $X$ of cardinality $m$. Fix a subset $C\subset X$ of cardinality $\ell$.
Let $C_{m-r}$ be the collection of subsets of $C$ of cardinality $m-r$ and let $(X\setminus C)_{r}$ be the collection of subsets of $X$ disjoint from $C$ and with cardinality $r$. We have a bijection $X_m\rightarrow \cup_{r=0}^m C_{m-r}\times (X\setminus C)_{r}$
given by 
$A\mapsto (A\cap C, A\cap (X\setminus C))$, which shows the equality. 
\end{proof}
\end{Lemma}

This proves the inequality \eqref{l increment} for any $\ell$.  

\begin{Remark}
    The induction hypothesis in \Cref{all equal factors} excludes the case $(2,2,2,5)$, meaning it is necessary to verify the case $(2,a+1,b+1,a+b+2)$ where $a\geq b\geq 2$. The inequality becomes $(a+b)^3\leq (a+b+3)(a+b-1)ab$, so it is enough to verify $(a+b)^2\leq (a+b-1)ab$. Rearranging the terms we obtain $a^2(b-1)+b^2(a-1)-3ab\geq0$, and if $a\geq b\geq 3$, it simplifies to $2(a^2+b^2)-3ab\geq 0$ since $(a-b)^2\geq 0$.
    
    Two cases are missing, namely $b=1$ and $b=2.$ In the first the inequality to be verified is $(a+1)^3\leq(a+4)a^2$. This is equivalent to $1\leq a(a+3)$, which is satisfied for $a\geq 4$.
    
    If $b=2$, the inequality becomes $(a+2)^2\leq 2(a+5)(a+1)a$. The polynomial $2(a+5)(a+1)a-(a+2)^2$ has roots in the intervals $(-7,-6), (-2,-1), (1,2)$ and its leading coefficient is positive, therefore it is positive for all $a\geq2$.

  Therefore the only cases where the inequality is not satisfied for $k\geq 3$ are $(2,2,2,5)$, $(2,2,3,6)$, $(2,2,4,7)$.
\end{Remark}

\begin{Corollary}
    Let $V$ be a tensor space of format $(n_1+1,n_2+1, n_3+1, n+2)$, where $n=n_1+n_2+n_3$, and let $T\in V$ general. If $(n_1,n_2,n_3)\in \{(1,1,1), (1,1,2), (1,1,3)\}$, then $\langle Z_T\rangle \subsetneq \PP(H_T)$.
\end{Corollary}

\section{Conjectures}

\begin{Conjecture}\label{conj:1}
The map $\alpha_T$ has maximal rank, i.e. 
\[
\mathrm{rank}(\alpha_T)=\min\left\{\prod_{j=1}^k\binom{n-1}{n_j}, (n+2)\prod_{j=1}^k\binom{n-2}{n_j}\right\}.
\]
\end{Conjecture}

The minimum on the right-hand side is determined in \Cref{sec:inequality}. 
The validity of this \Cref{conj:1} implies the following two conjectures for $3$ and $k+1\geq 4$ factors. 

\begin{Conjecture}\label{conj:2}
Let $T$ be a general tensor of format $(a+1,b+1,n+2)$, where $n=a+b$, and $a,b\geq2$. 
Then $\langle Z_T\rangle = \PP(H_T)$, except in the cases  $(3,3,6), (3,4,7),(3,5,8)$, where $\langle Z_T\rangle\subsetneq \PP(H_T)$
has codimension $\binom{a+b-1}{a}\binom{a+b-1}{b}-\binom{a+b-2}{a}\binom{a+b-2}{b}(a+b+2)$. 
\end{Conjecture}

\begin{Conjecture}\label{conj:3}
Let $n=\sum_{j=1}^kn_j$ and consider a $(k+1)$-tensor space $V$ of format $(n_1+1,\dots,n_k+1,n+2)$ and $T\in V$ a general tensor. If $k\geq 3$, then $\langle Z_T\rangle=\PP(H_T)$ except in the cases $(2,2,2,5)$, $(2,2,3,6)$, $(2,2,4,7)$. In these latter cases, $\langle Z_T\rangle\subsetneq \PP(H_T)$ has codimension 
$\prod_{j=1}^k\binom{n-1}{n_j}-(n+2)\prod_{j=1}^k\binom{n-2}{n_j}$.
\end{Conjecture}

\section{An interpretation of the map $\alpha_T$ in terms of Koszul cohomology}\label{sec:Koszul coh approach}

In this final and very short section, we explain the connections of Conjectures \ref{conj:2} and \ref{conj:3} to Koszul cohomology. For the ease of notation, we consider the $3$-factor case. Let $W$ be a vector space of dimension $m=a+b+2$, and let $A = \C^{a+1}$ and $B = \C^{b+1}$. Then a tensor $T\in W^*\otimes A\otimes B$ may be seen as a linear map
$W\rightarrow A\otimes B$. Since $T$ is generic, this linear map is injective and we may regard $W$ as a linear series $W\subset |\sO_X(1,1)|$, where $X = \PP(A)\times \PP(B)=\PP^a\times \PP^b$. Since $m>\dim(X)$, one has that the linear series $W$ is base-point free, which means that there exists a surjection
\begin{equation}\label{baseptfree}
W\otimes \sO_X\longrightarrow \sO_X(1,1)\longrightarrow 0. 
\end{equation}
Twisting this map of locally free  sheaves with $\sO_X(b-2,a-2)$ and taking global sections we obtain the dual of the desired map $\alpha_T$ between vector spaces. 

Consider the Koszul complex associated to the surjection \eqref{baseptfree} (upon twisting by $\sO_X(b-2,a-2)$). This is the complex $\sF_{\bullet}$, where
\[
\sF_i = \wedge^i W\otimes \sO_X(b-1-i,a-1-i), \mbox{ for }\ \  i=0,\ldots, m. 
\]
In this notation, recall that the map we are interested in is the dual of the map $\delta_1:H^0(\sF_1)\longrightarrow H^0(\sF_0)$. 
Since the complex $\sF_{\bullet}$ is exact, its hypercohomology 
$\mathbb{H}^{*}(\sF_{\bullet})=0$. Letting $E_1^{p,q}=H^q(\sF_p)$, the latter condition is equivalent to the vanishing $E^{p,q}_{\infty}=0$ for every $p,q\in \Z$. 
Consider now the following 3-complex
\begin{equation}\label{3complex}
H^{m-2}(\sF_m)\stackrel{\delta_m}{\longrightarrow}H^{m-2}(\sF_{m-1})\stackrel{\delta_{m-1}}{\longrightarrow}H^{m-2}(\sF_{m-2}).
\end{equation}
Note that the cohomology $H^i(\sF_j)$ vanishes for every $0<j<m-2=a+b$ and every $i$. Thus, on page $E_{a+b+1}=E_{m-1}$, we obtain a map 
\[
d^{m-1,m-2}_{m-1}: \ker(\delta_{m-1})/\mathrm{Im}(\delta_m)\longrightarrow \mathrm{coker}(\delta_1),
\]
where the left-hand side is the homology of the $3$-complex \eqref{3complex} and $\delta_1$ is the map on global sections introduced before. Since $E_{\infty}=0$, $d^{m-1,m-2}_{m-1}$ is an isomorphism. 
Now note that the $3$-complex \eqref{3complex} is dual to 
\[
\wedge^2 W\longrightarrow W\otimes (A\otimes B)\longrightarrow \sS^2 A\otimes \sS^2 B, 
\]
whose middle homology is $\mathrm{Tor}_1^S(R,\C)_{(2,2)}$, where 
$S = \sS^{\bullet}(W)$ and $R=\bigoplus_{d\geq 0} \sS^d A\otimes \sS^d B$ is the coordinate ring of the Segre variety $X$ and is naturally an $S$-module. The latter $\mathrm{Tor}$ $S$-module is by definition the Koszul cohomology module
$\mathcal K_{1,1}(X,\sO(1,1),W)$ \cite{Green84}.  

Choose a subspace $H\subset W$ of dimension $a+b+1=\dim(X)+1$ and define $\overline{S}=S/(H)\cong \C[z]$. Hence $\overline{R}=R/(H)$ is an Artinian reduction of $R$. One has the following isomorphism of syzygy modules
\[
\mathrm{Tor}^S_1(R,\C)_{(2,2)}\cong\mathrm{Tor}^{\overline{S}}_1(\overline{R},\C)_2=\ker(\overline{R}_1\stackrel{\cdot z}{\longrightarrow}
\overline{R}_2),
\]
where the last equality follows simply from the fact that the graded resolution 
of $\C$ over $\overline{S}$ is $\overline{S}(-1)\rightarrow \overline{S}$ and 
 $\mathrm{Tor}^{\overline{S}}_1(\overline{R},\C)_2$ is by definition the kernel 
 of the map $\overline{R}\otimes \overline{S}(-1)\cong \overline{R}(-1)\rightarrow \overline{R}\cong \overline{R}\otimes \overline{S}$ restricted in degree $2$.

\bibliographystyle{alpha}
\bibliography{bibliography}

@article{DHOST,
    AUTHOR = {Draisma, Jan and Horobe\c{t}, Emil and Ottaviani, Giorgio and
              Sturmfels, Bernd and Thomas, Rekha R.},
     TITLE = {The {E}uclidean distance degree of an algebraic variety},
   JOURNAL = {Found. Comput. Math.},
  FJOURNAL = {Foundations of Computational Mathematics. The Journal of the
              Society for the Foundations of Computational Mathematics},
    VOLUME = {16},
      YEAR = {2016},
    NUMBER = {1},
     PAGES = {99--149},
      ISSN = {1615-3375},
   MRCLASS = {14P05 (14N10 14Q15 90C26)},
  MRNUMBER = {3451425},
MRREVIEWER = {Nicolas Dutertre},
       DOI = {10.1007/s10208-014-9240-x},
       URL = {https://doi.org/10.1007/s10208-014-9240-x},
}

@article {FO,
    AUTHOR = {Friedland, Shmuel and Ottaviani, Giorgio},
     TITLE = {The number of singular vector tuples and uniqueness of best
              rank-one approximation of tensors},
   JOURNAL = {Found. Comput. Math.},
  FJOURNAL = {Foundations of Computational Mathematics. The Journal of the
              Society for the Foundations of Computational Mathematics},
    VOLUME = {14},
      YEAR = {2014},
    NUMBER = {6},
     PAGES = {1209--1242},
      ISSN = {1615-3375},
   MRCLASS = {15A69 (15A18 57R20 65F15 65H10 65K05)},
  MRNUMBER = {3273677},
MRREVIEWER = {Zhening Li},
       DOI = {10.1007/s10208-014-9194-z},
       URL = {https://doi.org/10.1007/s10208-014-9194-z},
}

@INPROCEEDINGS{L,
    author={ {Lek-Heng Lim}},
    booktitle={1st IEEE International Workshop on Computational Advances in Multi-Sensor Adaptive Processing, 2005.},
    title={Singular values and eigenvalues of tensors: a variational approach},   year={2005},
    volume={},
    number={},
    pages={129-132},
    doi={10.1109/CAMAP.2005.1574201},
    }

@article{turatti2021tensors,
author = {Turatti, Ettore},
title = {On Tensors That Are Determined by Their Singular Tuples},
journal = {SIAM Journal on Applied Algebra and Geometry},
volume = {6},
number = {2},
pages = {319-338},
year = {2022},
}

@book {GKZ,
    AUTHOR = {Gel'fand, Israel M. and Kapranov, Mikhail M. and Zelevinsky, Andrey V.},
     TITLE = {Discriminants, resultants, and multidimensional determinants},
    SERIES = {Mathematics: Theory \& Applications},
 PUBLISHER = {Birkh\"{a}user Boston, Inc., Boston, MA},
      YEAR = {1994},
     PAGES = {x+523},
      ISBN = {0-8176-3660-9},
   MRCLASS = {14N05 (13D25 14M25 15A69 33C70 52B20)},
  MRNUMBER = {1264417},
MRREVIEWER = {I. Dolgachev},
       DOI = {10.1007/978-0-8176-4771-1},
       URL = {https://doi.org/10.1007/978-0-8176-4771-1},
}

@article {DOT,
    AUTHOR = {Draisma, Jan and Ottaviani, Giorgio and Tocino, Alicia},
     TITLE = {Best rank-{$k$} approximations for tensors: generalizing {E}ckart-{Y}oung},
   JOURNAL = {Res. Math. Sci.},
  FJOURNAL = {Research in the Mathematical Sciences},
  volume={5},
  number={2},
  pages={27},
  year={2018},
      ISSN = {2522-0144},
   MRCLASS = {15A69 (15A18 15A72)},
  MRNUMBER = {3805895},
MRREVIEWER = {Gustavo Adolfo Mu\~{n}oz-Fern\'{a}ndez},
       DOI = {10.1007/s40687-018-0145-1},
       URL = {https://doi.org/10.1007/s40687-018-0145-1}
}

@article {OP,
    AUTHOR = {Ottaviani, Giorgio and Paoletti, Raffaella},
     TITLE = {A geometric perspective on the singular value decomposition},
   JOURNAL = {Rend. Istit. Mat. Univ. Trieste},
  FJOURNAL = {Rendiconti dell'Istituto di Matematica dell'Universit\`a di
              Trieste. An International Journal of Mathematics},
    VOLUME = {47},
      YEAR = {2015},
     PAGES = {107--125},
      ISSN = {0049-4704},
   MRCLASS = {15-02 (14N05 15A23 15A69)},
  MRNUMBER = {3456581},
MRREVIEWER = {Julio Ben\'{\i}tez},
       DOI = {10.13137/0049-4704/11222},
       URL = {https://doi.org/10.13137/0049-4704/11222}
}

@book{Weyman, place={Cambridge}, series={Cambridge Tracts in Mathematics}, title={Cohomology of Vector Bundles and Syzygies}, publisher={Cambridge University Press}, author={Weyman, Jerzy}, year={2003}, collection={Cambridge Tracts in Mathematics}}

@book {Ottaviani95,
    AUTHOR = {Ottaviani, Giorgio},
     TITLE = {Rational homogeneous varieties},
PUBLISHER = {people.dimai.unifi.it/ottaviani/rathomo/rathomo.pdf},
      YEAR = {1995},
     PAGES = {70},
}

@article {ST22,
    AUTHOR = {Sodomaco, Luca and  Turatti, Ettore Teixeira},
     TITLE = {The span of singular tuples of a tensor beyond the boundary
              format},
  JOURNAL = {J. Symbolic Comput.},
    VOLUME = {120},
      YEAR = {2024},
     PAGES = {Paper No. 102230, 19}
}

@article{HT25,
title = {When does subtracting a rank-one approximation decrease tensor rank?},
journal = {Linear Algebra and its Applications},
volume = {709},
pages = {397-415},
year = {2025},
author = {Emil Horobe\c{t} and Ettore Teixeira Turatti},
}

@article{RHST25,
author = {Ribot, {A}lvaro and Horobe\c{t}, Emil and Seigal, Anna and Turatti, Ettore T.},
title = {Decomposing Tensors via Rank-One Approximations},
journal = {SIAM Journal on Matrix Analysis and Applications},
volume = {47},
number = {1},
pages = {282-307},
year = {2026}
}

@article{RSZ,
      title={Orthogonal eigenvectors and singular vectors of tensors}, 
      author={Alvaro Ribot and Anna Seigal and Piotr Zwiernik},
      year={2025},
      journal={arXiv:2506.19009},
}

@article{WPKS,
    title={Multi-subspace power method for decomposing all tensors},
    author={Kexin Wang and Jo{\~{a}}o M. Pereira and Joe Kileel and Anna Seigal},
    year={2025},
    journal={arXiv:2510.18627},

}

@article {Green84,
    AUTHOR = {Green, Mark L.},
     TITLE = {Koszul cohomology and the geometry of projective varieties},
      NOTE = {With an appendix by Robert Lazarsfeld and Green},
   JOURNAL = {J. Differential Geom.},
  FJOURNAL = {Journal of Differential Geometry},
    VOLUME = {19},
      YEAR = {1984},
    NUMBER = {1},
     PAGES = {125--171},
      ISSN = {0022-040X,1945-743X},
   MRCLASS = {14F05 (14B12)},
  MRNUMBER = {739785},
MRREVIEWER = {G.\ Horrocks},
       URL = {http://projecteuclid.org/euclid.jdg/1214438426},
}

\end{document}